\documentclass[11pt, a4paper]{amsart}

\usepackage{amsthm,amsfonts,amsbsy,amssymb,amsmath,amscd}
\usepackage[all]{xy}
\usepackage[colorlinks=true]{hyperref}

\newcommand{\QQ}{\mathbb{Q}}
\newcommand{\RR}{\mathbb{R}}
\newcommand{\CC}{\mathbb{C}}
\newcommand{\PP}{\mathbb{P}}
\newcommand{\ZZ}{\mathbb{Z}}

\newcommand{\CA}{\mathcal{A}}

\newcommand{\CE}{\mathcal{E}}
\newcommand{\CH}{\mathcal{H}}
\newcommand{\CL}{\mathcal{L}}
\newcommand{\CO}{\mathcal{O}}
\newcommand{\CX}{\mathcal{X}}
\newcommand{\CY}{\mathcal{Y}}
\newcommand{\CZ}{\mathcal{Z}}

\newcommand{\charr}{\mathrm{char}\,}
\newcommand{\spec}{{\mathrm{Spec}}\,}

\newcommand{\ch}{{\mathrm{ch}}}

\newcommand{\nti}{{\mathrm{int}}}
\newcommand{\nth}{{\mathrm{th}}}

\newcommand{\rounddown}[1]{{\lfloor #1 \rfloor}}

\begin{document}
\title[Slope inequality]{Slope inequality for families of curves over surfaces}

\author{Tong Zhang}
\date{\today}

\address{Department of Mathematical Science, Durham University, Lower Mountjoy, Stockton Road, Durham DH1 3LE, United Kingdom}
\email{tong.zhang@durham.ac.uk}

\begin{abstract}
	In this paper, we investigate the general notion of the slope for families of curves $f: X \to Y$. The main result is an answer to the above question when $\dim Y = 2$, and we prove a lower bound for this new slope in this case over fields of any characteristic. Both the notion and the slope inequality are compatible with the theory for $\dim Y = 0, 1$ in a very natural way, and this gives a strong evidence that the slope for an $n$-fold fibration of curves $f: X \to Y$  may be $K_{X/Y}^n / \ch_{n-1}(f_* \omega_{X/Y})$.
	
	Rather than the usual stability methods, the whole proof of the slope inequality here is based on a completely new method using characteristic $p>0$ geometry. A simpler version of this method yields a new proof of the slope inequality when $\dim Y = 1$.
\end{abstract}

\maketitle


\theoremstyle{plain}
\newtheorem{theorem}{Theorem}[section]
\newtheorem{lemma}[theorem]{Lemma}
\newtheorem{coro}[theorem]{Corollary}
\newtheorem{prop}[theorem]{Proposition}
\newtheorem{defi}[theorem]{Definition}
\newtheorem{ques}[theorem]{Question}
\newtheorem{conj}[theorem]{Conjecture}

\newtheorem*{thm*}{Theorem}
\newtheorem*{conj*}{Conjecture}
\newtheorem*{ques*}{Question}

\theoremstyle{remark}
\newtheorem{remark}[theorem]{\bf Remark}
\newtheorem{assumption}[theorem]{\bf Assumption}
\newtheorem{example}[theorem]{\bf Example}

\numberwithin{equation}{section}



\section{Introduction}
In the study of families of curves, a fundamental notion is called the \emph{slope}. Let $f: X \to Y$ be a (non-isotrivial) fibration from a smooth surface $X$ to a smooth curve $Y$ with general fiber a smooth curve of genus $g \ge 2$. Assume that there is no $(-1)$-curve contained in fibers. Then the slope of $f$ refers to the following ratio:
$$
s(f) : = \frac{K^2_{X/Y}}{\deg f_* \omega_{X/Y}}.
$$
Throughout the past decades, it has become one of the key problems in algebraic geometry to investigate the bound of $s(f)$. These bounds are usually of great importance in many areas, including the slope conjecture itself (see \cite[Conjecture 0.1]{Harris_Morrison_Slope} for instance), the geography of surfaces (see \cite{Pardini_Severi}), the Oort conjecture (see \cite{Lu_Zuo_Oort_Conjecture,Chen_Lu_Zuo_Oort_Conjecture}). Among others, one most fundamental result related to $s(f)$ that has been essentially used in the aforementioned works is the following sharp lower bound:
\begin{equation}
	s(f) \ge \frac{4g-4}{g}. \label{surfaceslope}
\end{equation}
This bound was first obtained by Horikawa \cite{Horikawa_V} and Persson \cite{Persson_Double_Cover} for hyperelliptic fibrations over $\CC$ using a very explicit double cover method. The full version over $\CC$ was proved by Xiao \cite{Xiao_Slope} using the Harder-Narasimhan filtration for $f_*\omega_{X/Y}$, and independently by Cornalba-Harris \cite{Cornalba_Harris_Slope} for semi-stable fibrations via geometric invariant theory\footnote{Later, Stoppino \cite{Stoppino_Slope} found a way to treat non-semi-stable fibrations using the idea of Cornalba-Harris \cite{Cornalba_Harris_Slope}.}. Since then, this inequality is often referred as the \emph{slope inequality} or \emph{Cornalba-Harris-Xiao inequality} in the literature. Later, Moriwaki \cite{Moriwaki_Bogomolov} gave a different proof which works also in positive characteristics. We refer the reader to \cite[Chapter XIV]{ACG_Geometry_of_curves} as well as all references therein for more details about this slope inequality.

Not only the result itself, but also all methods mentioned above have been proven to be quite powerful. For example, both the Harder-Narasimhan filtration and geometric invariant theory have been applied in the study of various moduli problems (see \cite{Mumford_Stability,Huybrechts_Lehn_Moduli} for more details). 

It is fairly obvious that the above notion of the slope $s(f)$ only fits for one dimensional families of curves, because when $\dim Y > 1$, defining the degree of a vector bundle over $Y$ usually needs a polarization of which the choice is not that canonical unless some extra assumptions can be put. This naturally leads to the following general question:

\begin{ques*}
	What is the slope for a general family of curves $f: X \to Y$ where $\dim Y > 1$? Furthermore, is there a slope inequality in general?
\end{ques*}

The main purpose of this paper is to answer this question for $\dim Y = 2$, or equivalently, for $3$-fold fibrations of curves. Such an answer leads to a formulation of the slope of general families of curves.

\subsection{Main result}

Throughout this paper, $k$ always denotes an algebraically closed field. All varieties we consider are projective.

\begin{defi} \label{nfoldfibration}
	Let $f: X \to Y$ be a morphism between two normal varieties $X$ and $Y$ defined over $k$ with connected fibers. We say that $f$ is \emph{a fibration of curves of genus $g$}, if $f$ is flat, Cohen-Macaulay with pure relative dimension one and the general fiber of $f$ is a smooth curve of genus $g$.
\end{defi}

\begin{example}
	If $\dim X=1$, then $X$ is a normal hence smooth curve of genus $g$ over $Y=\spec k$. If $\dim X=2$ and in addition $X$ is smooth, $f$ is usually called a surface fibration in the previous literature. See \cite{BHPV} for example. Notice that the flatness and the Cohen-Macaulayness are guaranteed automatically in this case.
\end{example}

Suppose that $f: X \to Y$ is a fibration of curves of genus $g$ over $k$ as in Definition \ref{nfoldfibration}. By the duality theory, the relative dualizing sheaf $\omega_{X/Y}$ exists as a $Y$-flat and generically rank one sheaf on $X$, compatible with base changes (cf. \cite[Theorem 3.5.1, 3.6.1]{Conrad_Duality}). 
Furthermore, $f_* \omega_{X/Y}$ is a reflexive sheaf on $Y$ (see \cite[Corollary 5.26]{Horing_Positivity} for example), which particularly implies that $f_* \omega_{X/Y}$ is locally free if $Y$ is a smooth surface.

A very important property associated to $f_* \omega_{X/Y}$ is the \emph{positivity}. Much work related to this property in characteristic zero has been done by Fujita \cite{Fujita_Positivity}, Kawamata \cite{Kawamata_Characterization_of_Abelian}, Viehweg \cite{Viehweg_Weak_Positivity}, Koll\'ar \cite{Kollar_Subadditivity}, etc. For example, under Definition \ref{nfoldfibration}, if $\charr k = 0$, then $f_* \omega_{X/Y}$ is weakly positive. We refer the reader to \cite{Viehweg_Weak_Positivity} for the explicit definition of the weak positivity. In particular, $f_*\omega_{X/Y}$ is semi-positive as a locally free sheaf provided that $Y$ is a smooth surface. It turns out that this positivity becomes subtle in positive characteristics. Even for surface fibrations in positive characteristics, the weak positivity of $f_* \omega_{X/Y}$ could fail \cite[3.2]{Moret-Bailly_Families_de_courbes}. We refer to \cite{Patakfalvi_Semipositivity} for more details regarding this topic.

For any fibration $f: X \to Y$ from a normal variety $X$ to a Gorenstein variety $Y$, we write $\omega_{X/Y} = \omega_X \otimes f^* \omega^{-1}_Y$. In this case, $\omega_{X/Y}$ is a divisorial sheaf. We denote by $K_{X/Y}$ the relative canonical divisor associated to $\omega_{X/Y}$. We say that $f: X \to Y$ is \emph{relatively minimal}, if $K_{X/Y}$ is nef. Notice that if in addition $f$ is flat and Cohen-Macaulay, the sheaf $\omega_{X/Y}$ here coincides with the one obtained by the duality theory.

When $\charr k = 0$, the above definition of the \emph{relative minimality} is equivalent to that the canonical divisor $K_X$ is $f$-nef (cf. \cite[Theorem 3-1-1]{KMM_Intro_MMP}). However, this equivalence seems also subtle in positive characteristics. See \cite[Theorem 1.1]{Patakfalvi_Semipositivity} for more details.

Now we state the main theorem of this paper.

\begin{theorem}\label{main}
	Let $f: X \to Y$ be a relatively minimal fibration of curves of genus $g \ge 2$ from a normal $3$-fold $X$ to a smooth surface $Y$ defined over $k$ as in Definition \ref{nfoldfibration}. Then
	$$
	\left(\frac{1}{12} + \frac{1}{6g-6} \right)K^3_{X/Y} \ge  \ch_2(f_* \omega_{X/Y}),
	$$
	or equivalently
	$$
	K^3_{X/Y} \ge \frac{12g-12}{g+1} \ch_2(f_* \omega_{X/Y}),
	$$
	provided that one of the following assumptions holds:
	\begin{itemize}
		\item [(1)] $\charr k = 0$;
		\item [(2)] $\charr k > 0$ and $f_* \omega_{X/Y}$ is semi-positive.
	\end{itemize}
\end{theorem}
Here and in the following, for a vector bundle $\CE$, $\ch_i(\CE)$ refers to the $i^{\rm th}$ Chern character of $\CE$.

Let us put Theorem \ref{main} into perspective. Recall that for a relatively minimal fibration $f: X \to Y$ from a smooth surface $X$ to a smooth curve $Y$ over $k$ with general fiber a smooth curve of genus $g \ge 2$, we have the aforementioned slope inequality $s(f) \ge \frac{4g-4}{g}$ as (\ref{surfaceslope}). Notice that in this case, $\deg f_* \omega_{X/Y} = \ch_1(f_* \omega_{X/Y})$.
Thus the slope inequality (\ref{surfaceslope}) for $f$ just says the following:
$$
\left(\frac{1}{4} + \frac{1}{4g-4} \right) K^2_{X/Y} \ge \ch_1(f_* \omega_{X/Y}).
$$

In fact, such type of results also exists for a single curve (zero dimensional family of curves). Suppose that $\dim X = 1$, i.e., $f: X \to \spec k$ is a smooth curve of genus $g \ge 2$. In this case, $h^0(K_{X/k}) = \ch_0(f_* \omega_{X/k})$. Then the ``slope" (in)equality for $f$ is
$$
\left(\frac{1}{2} + \frac{1}{2g-2} \right) \deg K_{X/k} = \ch_0(f_* \omega_{X/k}),
$$
which is nothing but Riemann-Roch for $K_{X/k}$.

Theorem \ref{main} generalizes both of the above results to $3$-fold fibrations of curves of genus $g \ge 2$. Moreover, their compatibility suggests that for a relatively minimal $n (\ge 4)$-fold fibration $f: X \to Y$ of curves as in Definition \ref{nfoldfibration}, a suitable notion of the slope is likely to be the ratio between the following two invariants:
$$
K^n_{X/Y} \quad \mbox{and} \quad \ch_{n-1}(f_* \omega_{X/Y}).
$$
In a forthcoming paper \cite{Zhang_Relative_Clifford_inequality}, we will investigate this proposed slope when $n \ge 4$. In particular, we will prove that there is a slope inequality between the above two invariants when $Y$ is an abelian variety of arbitrary dimension in characteristic zero which is compatible with all known results.

\subsection{Idea of the proof}
Roughly speaking, our proof of Theorem \ref{main} is based on a characteristic $p>0$ method. Namely, we first prove Theorem \ref{main} in positive characteristics. Then the characteristic zero result follows from a weaker version of Theorem \ref{main} in positive characteristics and a perturbation method.

Recall that the methods used by Xiao \cite{Xiao_Slope}, Cornalba-Harris \cite{Cornalba_Harris_Slope} and Moriwaki \cite{Moriwaki_Bogomolov} to prove (\ref{surfaceslope}) are all stability methods in various versions, but our proof here does not employ any stability method. Moreover, the method in this paper does allow us to give a new proof of (\ref{surfaceslope}) in arbitrary characteristic, and this will be presented at the end of this paper.\footnote{It is not clear, at least to the author, that whether the aforementioned stability methods could give rise to a slope inequality over surfaces other than curves.}

For the convenience of the reader, in the following, we explain our proof in more details.

\subsubsection*{Step 1: $F$-stable dimension in characteristic $p > 0$}
Recall that Theorem \ref{main} is a result about the self-intersection number of divisors on $3$-folds and the Chern character of vector bundles over surfaces. In this paper, we introduce a new notion in characteristic $p > 0$, namely the \emph{$F$-stable dimension of global sections}, serving as a bridge connecting the above two invariants.

The formal definition is the following.
\begin{defi} \label{h0f}
	Let $V$ be a variety over $k$ of characteristic $p > 0$. Let $F^e$ be the $e^{\rm th}$ absolute Frobenius morphism of $V$. For any coherent sheaf $\CE$ on $V$, the \emph{$F$-stable dimension of global sections $h^0_F(\CE)$} of $\CE$ is defined as 
	$$
	h^0_F(\CE) : = \liminf_{e \to \infty} \frac{h^0(F^{e*} \CE)}{p^{e \dim V}}.
	$$
\end{defi}

This new invariant is not easy to compute in general, but when $\CE$ is a semi-positive vector bundle over a surface in positive characteristics, we can show that
\begin{equation}
h^0_F(\CE) \ge \ch_2(\CE). \label{reduction0}
\end{equation}
See Proposition \ref{h0ch2}. Therefore, to prove Theorem \ref{main} in positive characteristics, it suffices to prove the following inequality:
\begin{equation}
\left(\frac{1}{12} + \frac{1}{6g-6} \right)K^3_{X/Y} \ge h^0_F (f_* \omega_{X/Y}). \label{reduction1}
\end{equation}
This new inequality (\ref{reduction1}) is a comparison between the intersection number and the dimension of global sections, although the right hand side is a limit. However, this reduction somehow indicates that we may prove Theorem \ref{main} in positive characteristics by
\begin{itemize}
	\item [(i)] setting up a rough estimate first, and then
	\item [(ii)] taking the limit using the Frobenius morphism.
\end{itemize}
In fact, the proof does go in this way.

\subsubsection*{Step 2: An estimate via double filtrations}
Now we generalize the problem a bit. Let $\charr k = p > 0$ and $f: X \to Y$ be a fibration from a normal $3$-fold $X$ to a smooth surface $Y$ defined over $k$ with general fiber $C$ a smooth curve of genus $g \ge 2$, not necessarily flat or Cohen-Macaulay. Let $L$ be a nef $\QQ$-divisor on $X$ such that $L|_C \ge K_C \ge \rounddown{L|_C}$. According to Step 1, we need a rough estimate as
\begin{equation}
	\left(\frac{1}{12} + \frac{1}{6g-6} \right)L^3 \ge h^0(\CO_X(\rounddown{L})) + \mbox{extra term} \label{reduction2}
\end{equation}
such that the ``Frobenius limit" of the extra term is zero.

The method for proving (\ref{reduction2}) is based on a double filtration for nef divisors. Replacing $Y$ by an appropriate blowing up, we may assume that $Y$ is also fibered over a smooth curve $B$. Then we get a $2$-tower of fibrations from $X$ to $B$ as follows:
$$
\xymatrix{
	X \ar[r]^f \ar@/^1.5pc/[rr]^{\pi} & Y \ar[r] & B
	}
$$
Replacing $X$ by an appropriate blowing up, we can construct a filtration 
$$
L := L_0 > L_1 > \ldots > L_N \ge 0
$$
of nef divisors on $X$ with respect to the fibration $\pi$. Denote by $F$ a general fiber of $\pi$. The key point is that $f|_F$ is also a fibration. Therefore, for each $L_i|_F$, we can construct a similar filtration
$$
L_i|_F := L_{i, 0} > L_{i, 1} > \ldots > L_{i, N_i} \ge 0
$$
of nef divisors on $F$ with respect to $f|_F$. Then we are able to compare $L^3$ and $h^0(\CO_X(\rounddown{L}))$ using the nef thresholds (see Definition \ref{nefthresholddef}) obtained from this double filtration, and the proof of (\ref{reduction2}) can be completed in this way. We refer to Theorem \ref{relnoether32} for an explicit version of (\ref{reduction2}).

\subsubsection*{Step 3: Limiting method via the Frobenius base change}
The idea to prove (\ref{reduction1}) from (\ref{reduction2}) is via the Frobenius base change. Let $f: X \to Y$ be as in Theorem \ref{main}. Let $L$ be any nef $\QQ$-divisor as in Step 2. By the base change of $f$ via the $e^{\nth}$ Frobenius morphism $F^e$, we get a new fibration $f_e: X_e \to Y$ where $X_e = X \times_{F^e} Y$. Let $L_e$ be the pullback of $L$ to $X_e$. We can apply (\ref{reduction2}) to get an inequality between $L^3_e$ and $h^0(f_{e*}\CO_{X_e} (\rounddown{L_e}))$. Finally,  (\ref{reduction1}) follows after we replace $L$ by $K_{X/Y}$ and take $e \to \infty$. 

\subsubsection*{Step 4: Weaker results which imply Theorem \ref{main} in characteristic zero} Now we go to the proof of Theorem \ref{main} when $\charr k = 0$. The basic idea is to use the reduction mod $p$. We could directly apply (\ref{reduction0}) and (\ref{reduction1}) to get the desired result. However, there are some obstructions. To explain them, let $\CX \to \CY \to \CZ$ be an integral model extending $X \to Y \to \spec k$. Although $K_{X/Y}$ is nef and $f_* \omega_{X/Y}$ is semi-positive, unfortunately, we do not know whether we are able to find a closed point $z \in \CZ$ so that $K_{\CX_z / \CY_z}$ is nef and ${f_z}_* \omega_{\CX_z /\CY_z}$ is semi-positive. In fact, there are several conjectures about whether the nefness can be kept after the reduction mod $p$. See the problem \cite[Problem 5.4]{Miyaoka_Chern} posed by Miyaoka and a more general conjecture by Langer in \cite[Conjecture 5.5]{Langer_Positivity} for example. Recently, Langer \cite{Langer_adv} has constructed examples of nef divisors in characteristic zero whose each reduction is not nef, which is somehow surprising. Nevertheless, by \cite[Chap. III, Theorem 4.7.1]{Grothendieck_EGA3}, we can keep the ampleness via the reduction mod $p$. Notice that the nefness is indeed a limit of the ampleness. This suggests that we may perturb both $K_{X/Y}$ and $f_* \omega_{X/Y}$ a bit to get a weaker result first and then use this weaker result to approximate Theorem \ref{main}.

Our proof goes as follows. We find that in order to prove Theorem \ref{main}, it suffices to prove
\begin{equation}
\left(\frac{1}{12} + \frac{1}{6g-6} \right)(K_{X/Y} + \varepsilon A)^3 \ge \ch_2(f_* \omega_{X/Y}) \label{reduction3}
\end{equation}
for an ample divisor $A \ge 0$ on $X$ and for any $\varepsilon > 0$ sufficiently small. Although weaker than Theorem \ref{main}, (\ref{reduction3}) does imply Theorem \ref{main} when taking $\varepsilon \to 0$. A very important advantage of (\ref{reduction3}) is that $K_{X/Y} + \varepsilon A$ is ample. This allows us to use the reduction mod $p$ method. 

Let $\CA$ be the universal line bundle on $\CX$ extending $A$. Fix an $\varepsilon > 0$ very small. Similar to Step $2$ and $3$, we are able to prove that 
\begin{equation}
\left(\frac{1}{12} + \frac{1}{6g-6}\right)(K_{\CX_z/\CY_z} + \varepsilon \CA_z)^3 \ge h^0_F({f_z}_* \omega_{\CX_z/\CY_z}) \label{reduction4}
\end{equation}
for almost $z \in \CZ$. It can be seen from here that in order to get (\ref{reduction4}), we do need to work with $\QQ$-divisors. This illustrates the necessity why in Step 2 we have to deal with $\QQ$-divisors other than integral divisors only.

Now the issue is to compare $h^0_F(f_* \omega_{\CX_z/\CY_z})$ and $\ch_2(f_{z*} \omega_{\CX_z/\CY_z})$. We do not know whether ${f_z}_* \omega_{\CX_z/\CY_z}$ is still semi-positive to ensure (\ref{reduction0}). However, since ${f_z}_* \omega_{\CX_z/\CY_z}$ is very close to be semi-positive, we are able to prove that for any fixed $\varepsilon' > 0$, the inequality
\begin{equation}
h^0_F({f_z}_* \omega_{\CX_z/\CY_z}) \ge \ch_2({f_z}_* \omega_{\CX_z/\CY_z}) - \varepsilon' \label{reduction5}
\end{equation}
holds for almost all $z \in \CZ$. Here we use the perturbation method once more, twisting $f_*\omega_{X/Y}$ by ample $\QQ$-divisors before reduction mod $p$. Since $\varepsilon'$ is arbitrary, (\ref{reduction4}) and (\ref{reduction5}) together will imply (\ref{reduction3}). Hence the proof of Theorem \ref{main} in characteristic zero is completed.

\subsection{Structure of the paper} This paper is organized as follows. In Section $2$, we introduce some basic notions and results. Section $3$--$6$ are devoted to the study of fibered surfaces and $3$-folds, the construction of the double filtration and the proof of (\ref{reduction2}). In Section $7$, we focus on vector bundles over surfaces in positive characteristics, and the main results there are (\ref{reduction0}) and (\ref{reduction5}). Eventually, the proof of Theorem \ref{main} is presented in Section $8$, where we also present an example indicating that the inequality in Theorem \ref{main} is very close to be optimal. In Section $9$, adopting the whole idea above, we give a completely new (stability-free) proof of the slope inequality (\ref{surfaceslope}).

\subsection*{Acknowledgment} This paper was conceived during the author's visit to Johannes Gutenberg Universit\"at Mainz in January 2015. The author would like to thank Professor Kang Zuo for sharing his perspectives on the slope inequality and its connection to other areas. The author would like to thank Professors Xi Chen, Yifei Chen, Christopher Hacon, Adrian Langer, Zsolt Patakfalvi, Karl Schwede, Lidia Stoppino, Xiaotao Sun, Xiaowei Wang, Chenyang Xu and Xinyi Yuan, Doctors Jianke Chen, Xiaolei Liu and Jinsong Xu for their comments on this paper. Special thanks go to Professor Xiaotao Sun for his communications with the author on this topic since 2013 as well as pointing out some mistakes in the first draft of the current paper, and to Professor Adrian Langer for bringing the paper \cite{Langer_Moduli} to the author's attention.

Part of this paper was presented and written during the author's visit to the University of Hong Kong, the University of Utah, Beijing International Center for Mathematical Research, Mathematical Institute in Chinese Academy of Science and East China Normal University. The author would like to thank their hospitalities.

\section{Preliminaries}

In this section, we list some conventions and fundamental results which will be used throughout this paper.

\subsection{Conventions}
In this paper, we follow the conventions below.

\subsubsection*{Fibration} Let $V_1$ and $V_2$ be two varieties. A fibration $f: V_1 \to V_2$ means a surjective morphism from $V_1$ to $V_2$ with connected fibers.

\subsubsection*{Divisor} Let $V$ be a variety. We say that $L$ is a $\QQ$-divisor on $V$, if $L = \sum a_i D_i$ where $a_i \in \QQ$ and each $D_i$ is a prime divisor on $V$. We denote by $\rounddown{L}$ the integral part of $L$, i.e., $\rounddown{L} = \sum \rounddown{a_i} D_i$ where $\rounddown{a_i}$ is the biggest integer not exceeding $a_i$. A divisor $L$ is $\QQ$-Cartier if certain multiple of $L$ is Cartier. If $V$ is smooth, we will not distinguish integral divisors and line bundles, and we will simply denote $h^0(\CO_V(L))$ by $h^0(L)$ for $L$ integral.

\subsubsection*{Vector bundle} Let $V$ be a variety and $\CE$ be a vector bundle over $V$. We denote by $\PP(\CE)$ the projectivization of $\CE$ and by $\pi_{\CE}$ the projection $\PP(\CE) \to V$. We also denote by $H_\CE$ a divisor associated to $\CO_{\PP(\CE)}(1)$. Notice that $\CE$ is positive (resp. semi-positive) if $H_{\CE}$ is ample (resp. nef).

\subsection{Nef threshold with respect to fibrations over curves} Let $f: X \to B$ be a fibration from a variety $X$ to a smooth curve $B$ with a general fiber $F$ integral.

\begin{defi} \label{nefthresholddef}
	Let $L$ be a nef divisor on $X$. The \emph{nef threshold of $L$ with respect to $f$} is the following real number:
	$$
	\mathrm{nt}_f(L) := \sup \{a \in \RR | L- aF \, \mbox{is nef} \}.
	$$
\end{defi}
In this paper, we mainly consider the following larger integral invariant
$$
\nti_{f}(L) := \rounddown{\mathrm{nt}_f(L)} + 1 = \min \{a \in \ZZ | L- aF \, \mbox{is not nef} \}.
$$
This notion has been used in \cite{YuanZhang_RelNoether,Zhang_Severi}. It is easy to see that $\nti_f(L) \ge 1$. Moreover, we have the following simple result about this invariant.

\begin{prop} \label{nefthreshold}
	Suppose that $\sigma: X' \to X$ is a birational morphism. Let $f': X' \stackrel{\sigma}{\to} X \stackrel{f}{\to} B$ be the induced fibration with a general fiber $F'=\sigma^* F$. Then for any nef divisor $L$ on $X$,
	$$
	\nti_{f'}(\sigma^*L) = \nti_f(L).
	$$
\end{prop}

\begin{proof}
	This result follows from the observation that for any $a \in \ZZ$, $L-aF$ is nef if and only if $\sigma^*L - aF'$ is nef.
\end{proof}

\subsection{Remark on intersection numbers}
In this paper, the computation of intersection numbers appears frequently. We will use the following result all the time and sometimes may not mention it.

\begin{prop} \label{intersectionnumber}
	Let $V$ be a variety of dimension $n$. Let $A_1$, $\cdots$, $A_n$, $B_1$, $\cdots$, $B_n$ all be nef $\QQ$-divisors on $V$ such that $B_i-A_i$ is pseudo-effective for any $1 \le i \le n$. Then we have
	$$
	A_1A_2 \cdots A_n \le B_1B_2 \cdots B_n.
	$$
\end{prop}

\begin{proof}
	Since $B_1-A_1$ is pseudo-effective, by the nefness assumption, we deduce that
	$$
	A_1 A_2 \cdots A_n \le B_1 A_2 \cdots A_n.
	$$
	Thus by induction, the proof is completed.
\end{proof}

\section{Linear system on fibered surfaces}

Let $f: X \to B$ be a fibration from a smooth surface $X$ to a smooth curve $B$ defined over $k$ such that the general fiber $F$ is a smooth curve of genus $g \ge 2$. Let $L \ge 0$ be a nef divisor on $X$. Denote $a_0 = \nti_f(L)$.

\begin{theorem} \label{filtration2}
	Suppose that $LF > 0$. Then we have the following sequence of triples
	$$
	\{(L_i, Z_i, a_i) | i=0, \ldots, N\}
	$$
	on $X$ such that 
	\begin{itemize}
		\item $(L_0, Z_0, a_0) = (L, 0, \nti_f(L))$.
		\item For any $i \ge 1$, we have the decomposition
		$$
		|L_{i-1}-a_{i-1}F| = |L_i| + Z_i,
		$$
		where $Z_i$ is the fixed part of $|L_{i-1}-a_{i-1}F|$ and the movable part $L_i$ of $|L_{i-1}-a_{i-1}F|$ is nef whose base locus (if not empty) has dimension zero. Here $a_i = \nti_f(L_i)$.
		\item $h^0(L_N-a_NF) = 0$.
		\item $L_0F > L_1F > \ldots > L_NF \ge 0$.
	\end{itemize}
\end{theorem}

\begin{proof}
	See \cite[Theorem 2.2]{YuanZhang_RelNoether}.
\end{proof}

\begin{remark} \label{surfaceri}
	A direct consequence of Theorem \ref{filtration2} is that $|L_i|_F|$ is base point free for any $i > 0$. Moreover, if $|L_0|_F| = |L|_F|$ is also base point free, then
	$$
	h^0(L_0|_F) > h^0(L_1|_F) > \cdots > h^0(L_N|_F) \ge 1.
	$$
\end{remark}

All the notation in Theorem \ref{filtration2} will be used in this section. 

\begin{prop} \label{numerical2}
	Suppose that $LF > 0$. Then we have the following numerical inequalities:
	\begin{itemize} 
		\item[(1)] $\displaystyle h^0(L)  \le  \sum_{i=0}^{N} a_i h^0(L_i|_F)$; 
		\item[(2)] $\displaystyle L^2 \ge  2(a_0-1) L_0F + \sum_{i=1}^{N} a_i(L_{i-1}F+L_iF)$.
	\end{itemize}
\end{prop}

\begin{proof}
	See \cite[Proposition 2.3]{YuanZhang_RelNoether}.
\end{proof}

The main theorem in this section is the following:

\begin{theorem} \label{surface}
	Suppose that  $|L|$ is base point free. 
	\begin{itemize}
		
		\item [(1)] If $LF=0$, then 
		$$
		h^0(L) \le a_0.
		$$
		\item [(2)] If $LF > 0$, then
		$$
		h^0(L) \le \frac{1}{2} L^2 + LF.
		$$
		\item [(3)] If $0 \ne L|_F \le K_F$, then
		$$
		h^0(L) - \frac{1}{4} L^2 \le a_0 + \frac{1}{2} \sum_{i=1}^{N} a_i + \frac{LF}{2}.
		$$
		\item [(4)] If $F$ is non-hyperelliptic and  $0 \ne L|_F < K_F$, then
		$$
		h^0(L) - \frac{1}{4} L^2 \le \frac{1}{2} a_0 + \frac{1}{4} a_N + \frac{LF}{2}.
		$$
	\end{itemize}
\end{theorem}

\begin{proof}
	We prove (1) first. Since $LF=0$ and $|L|$ is base point free, we may assume that $L \sim_{\mathrm{num}} rF$ for some $r \ge h^0(L) - 1$. Notice that $a_0=\nti_{f}(L)$. By Definition \ref{nefthresholddef}, we know that
	$$
	a_0 = r + 1 \ge h^0(L).
	$$
	Hence (1) is proved.
	
	From now on, we always assume that $LF>0$. Apply Theorem \ref{filtration2} to $L$, and we can get triples
	$$
	\{(L_i, Z_i, a_i) | i=0, \ldots, N\}
	$$
	on $X$ which satisfy the conditions therein. Write
	$$
	r_i = h^0(L_i|_F) \quad \mbox{and} \quad d_i=L_iF.
	$$
	By Theorem \ref{filtration} and Remark \ref{surfaceri}, we know that
	\begin{equation}
		r_0 > r_1 > \cdots > r_N \ge 1. \label{ri>2}
	\end{equation}
	Since $g \ge 2$, by Riemann-Roch, we have
	$$
	r_i \le \left\{ 
	\begin{array} {ll}
	d_i, & \mbox{if \,} i < N; \\
	d_i + 1, & \mbox{if \,} i=N.
	\end{array}
	\right.
	$$
	Therefore, by Proposition \ref{numerical2} and (\ref{ri>2}), we deduce that
	\begin{eqnarray*}
		h^0(L) & \le &  a_0r_0 + \frac{1}{2} \sum_{i=1}^{N} a_i (r_{i-1} + r_i - 1) \\
		& \le & a_0d_0 + \frac{1}{2}\sum_{i=1}^{N} a_i (d_{i-1} + d_i) \\
		& \le & \frac{1}{2} L^2 + LF,
	\end{eqnarray*}
	which is exactly (2).
	
	To prove (3) and (4), we assume that $0 < L|_F \le K_F$. Recall that the standard Clifford's inequality asserts that for any $i$,
	\begin{equation}
	r_i \le \frac{1}{2} d_i + 1. \label{clifford}
	\end{equation}
	Thus Proposition \ref{numerical2}, (\ref{ri>2}) and (\ref{clifford}) imply that
	\begin{eqnarray*}
		h^0(L) & \le &  a_0r_0 + \frac{1}{2} \sum_{i=1}^{N} a_i (r_{i-1} + r_i - 1) \\
		& \le & \left(\frac{1}{2} d_0 + 1 \right)a_0 + \frac{1}{2} \sum_{i=1}^{N} a_i \left(\frac{1}{2}d_{i-1} + 1 + \frac{1}{2} d_i \right) \\
		& = & \frac{1}{2} a_0d_0 + \frac{1}{4} \sum_{i=1}^{N} a_i(d_{i-1}+d_i) + a_0 + \frac{1}{2}\sum_{i=1}^{N} a_i \\
		& \le & \frac{1}{4} L^2 + \frac{LF}{2} + a_0 + \frac{1}{2}\sum_{i=1}^{N} a_i.
	\end{eqnarray*}
	Thus the proof of (3) is completed.
	
	If $F$ is non-hyperelliptic and $L_F < K_F$, we have a shaper Clifford's inequality
	\begin{equation}
	r_i \le \frac{1}{2} d_i + \frac{1}{2} = \frac{1}{2} (d_i + 1) \label{clifford1}
	\end{equation}
	for $0 \le i \le N-1$. Combine this with Proposition \ref{numerical2} and (\ref{ri>2}). Similar to the above, it follows that
	\begin{eqnarray*}
		h^0(L) & \le &  a_0r_0 + \frac{1}{2} \sum_{i=1}^{N} a_i (r_{i-1} + r_i - 1) \\
		& \le & \frac{1}{2} (d_0 + 1) a_0 + \frac{1}{4} \sum_{i=1}^{N-1}a_i (d_{i-1} + d_i) + \frac{1}{4} a_N (d_{N-1} + d_N + 1) \\
		& \le & \frac{1}{4} L^2 + \frac{LF}{2} + \frac{1}{2}a_0 + \frac{1}{4} a_N.
	\end{eqnarray*}
	This proves (4), and the whole proof is completed.
\end{proof}

\begin{prop} \label{pl}
	Suppose that $P$ is a nef $\QQ$-divisor on $X$ such that $P|_F \ge K_F$. Then
	$$
	\frac{PL}{2g-2} \ge \sum_{i=0}^{N} a_i - 1.
	$$
\end{prop}

\begin{proof}
	From Theorem \ref{filtration2}, we know that
	$$
	L = L_0 \ge \left(\sum_{i=0}^{N} a_i - 1 \right)F + L_N-(a_N-1)F.
	$$
	Moreover, since $a_N=\nti_f(L_N)$, we know that $L_N-(a_N-1)F$ is nef. Thus it follows that 
	$$
	PL \ge \left(\sum_{i=0}^{N} a_i - 1 \right)PF + \left(L_N - (a_N-1) F \right) P \ge \left(\sum_{i=0}^{N} a_i - 1 \right) PF.
	$$
	The proof is completed as $PF \ge 2g-2$.
\end{proof}

\section{Filtration for nef divisors on fibered varieties over curves}

In Section 3, for any nef divisor $L$ on fibered surfaces, we are able to construct a filtration of nef divisors. Using this filtration, we can estimate $h^0(L)$, $L^2$ and $PL$ respectively. In this section, we generalize these results for surfaces to higher dimensions, particularly, to dimension $3$.

Let $f: X \to B$ be a fibration from a smooth variety $X$ of dimension $n (\ge 3)$ to a smooth curve $B$ defined over $k$ such that the general fiber $F$ is integral, not necessarily smooth. Let $L \ge 0$ be a nef divisor on $X$.

\subsection{The filtration}
Similar to the surface case, by \cite[Theorem 2.3]{Zhang_Severi}, there is a sequence of quadruples
$$
\{(X_i, L_i, Z_i, a_i), \quad i=0, 1, \cdots, N\}
$$
which satisfy the following three conditions:
\begin{itemize}
	\item [(C1)] $(X_0, L_0, Z_0, a_0)=(X, L, 0, \nti_f(L))$.
	\item [(C2)] For any $i=0, \cdots, N-1$, $\pi_i: X_{i+1} \to X_i$ is the blowing up of the base locus of $|L-a_iF_i|$. Here $F_0=F$, $F_{i+1} = \pi^*_iF_i$ and $a_{i+1} = \nti_{f_{i+1}}(L_{i+1})$ where $f_{i+1}: X_{i+1} \stackrel{\pi_i}{\rightarrow} \cdots \stackrel{\pi_0}{\rightarrow} X_0 \stackrel{f}{\rightarrow} B$ is the induced fibration. In particular, we have the following decomposition 
	$$
	\pi_i^*|L_i-a_iF_i| = |L_{i+1}| + Z_{i+1},
	$$
	where $Z_{i+1}$ is the base locus of $\pi_i^*|L_i-a_iF_i|$ and the movable part $|L_{i+1}|$ of $\pi_i^*|L_i-a_iF_i|$ is base point free.
	\item [(C3)] $h^0(L_N-a_NF_N)=0$.
\end{itemize}
We remark here that although \cite[Theorem 2.3]{Zhang_Severi} only concerns the case of characteristic zero, it is easy to check that this result also holds in positive characteristics. For the $n=3$ case in which we are mostly interested, this can be ensured also by the resolution of singularities \cite{Cossart_Piltant_Resolution1,Cossart_Piltant_Resolution2,Cutkosky_Resolution}.

Unlike the surface case, the above quadruples are on different smooth (birational) models of $X$. In order to make all quadruples be on one model as in Theorem \ref{filtration2}, we state the following theorem which can be viewed as a generalization of Theorem \ref{filtration2}, slightly better than \cite[Theorem 2.3]{Zhang_Severi}.

\begin{theorem} \label{filtration}
	Let $f: X \to B$, $F$ and $L$ be as above. Then there is a birational morphism $\sigma: X_L \to X$ and a sequence of triples
	$$
	\{(L_i, Z_i, a_i), \quad i=0, 1, \cdots, N\}
	$$
	on $X_L$ with the following properties:
	\begin{itemize}
		\item $(L_0, Z_0, a_0)=(\sigma^*L, 0, \nti_{f_L}(L_0))$ where $f_L: X_L \stackrel{\sigma}{\to} X \stackrel{f}{\to} B$ is the induced fibration. 
		\item For any $i=0, \cdots, N-1$, we have the decomposition
		$$
		|L_i-a_iF_L|= |L_{i+1}| + Z_{i+1},
		$$
		where $Z_i \ge 0$ is the fixed part of $|L_i-a_iF_i|$ and the movable part $|L_{i+1}|$ of $|L_i-a_iF_i|$ is base point free. Here $F_L = \sigma^*F$ is a general fiber of $f_L$, and $a_{i+1} = \nti_{f_L}(L_{i+1})$.
		\item $h^0(L_N-a_NF_N)=0$.
	\end{itemize}
\end{theorem}

\begin{proof}
	By \cite[Theorem 2.3]{YuanZhang_RelNoether}, we can get the following quadruples 
	$$
	\{(X_i, L_i, Z_i, a'_i), \quad i=0, 1, \cdots, N\}
	$$
	which satisfy the condition (C1) -- (C3). In particular, $a'_i = \nti_{f_i}(L_i)$.
	
	Now take $X_L=X_N$ and replace $L_i$ and $Z_i$ by their pullbacks onto $X_L$, by which we get a sequence of triples
	$$
	\{(L_i, Z_i, a'_i), \quad i=0, 1, \cdots, N\}
	$$
	on $X_L$. We will show that these triples satisfy the required properties. 
	
	In fact, it suffices to show that 
	$$
	a'_i = \nti_{f_L}(L_i)
	$$
	for any $i$. However, this equality simply follows from Proposition \ref{nefthreshold}. Hence the proof is completed.
\end{proof}

\subsection{Numerical inequalities}
For any nef divisor $L \ge 0$ on $X$, by Theorem \ref{filtration}, we obtain a sequence of triples
$$
\{(L_i, Z_i, a_i), \quad i=0, 1, \cdots, N\}
$$
on a birational model $X_L$ of $X$. For simplicity, we still denote by $F$ a general fiber of $f_L: X_L \to B$.

\begin{prop} \label{h0}
	We have
	$$
	h^0(L) \le \sum_{i=0}^{N} a_i h^0(L_i|_F).
	$$
\end{prop}

\begin{proof}
	See \cite[Proposition 2.6 (1)]{Zhang_Severi} in the characteristic zero case. Notice that the proof there applies in positive characteristics verbatim.
\end{proof}

For any $0 \le i \le N$, write
$$
L'_i := L_i-(a_i-1)F.
$$
It is easy to see from (\ref{nefthresholddef}) that $L'_i$ is nef on $X_L$ and 
$$
L'_i = a_{i+1} F + Z_{i+1}+L'_{i+1}.
$$

From now on, we always assume that $n = \dim X = 3$.
\begin{prop} \label{intersection31}
	For any $1 \le i \le N$, we have
	\begin{eqnarray*}
		L^3 - L'^3_i & = &  3(a_0-1)L^2_0F + \sum_{j=1}^{i} a_i(L^2_{j-1} + L_{j-1}L_{j} + L^2_j)F \\
		& & + \sum_{j=1}^{i} (L'^2_{j-1} + L'_{j-1}L'_{j} + L'^2_{j})Z_j.
	\end{eqnarray*}
\end{prop}

\begin{proof}
	This is a direct computation. For any $0 \le j \le i-1$, we have
	\begin{eqnarray*}
		L'^3_j - L'^3_{j+1} & = & (L'_j-L'_{j+1})(L'^2_j +L'_jL'_{j+1}+L'^2_{j+1}) \\
		& = & (a_{j+1}F + Z_{j+1})(L'^2_j+L'_jL'_{j+1}+L'^2_{j+1}).
	\end{eqnarray*}
	Since $L_j|_F = L'_j|_F$, it is easy to check that 
	$$
	(L'^2_j+L'_jL'_{j+1}+L'^2_{j+1})F = (L^2_j+L_jL_{j+1}+L^2_{j+1})F.
	$$
	Summing up for all $j$, it follows that
	$$
	L'^3_0 - L'^3_N = \sum_{j=1}^{N} a_j(L^2_{j-1} + L_{j-1}L_{j} + L^2_j)F + \sum_{j=1}^{N} (L'^2_{j-1} + L'_{j-1}L'_{j} + L'^2_{j})Z_j.
	$$
	Notice that 
	$$
	L^3_0 - L'^3_0 = 3(a_0-1)L^2_0F.
	$$
	Hence the proof is completed.
\end{proof}

\begin{prop} \label{intersection32}
	Let $P$ be a nef $\QQ$-divisor on $X$ and $P_0= \sigma^* P$. Then
	$$
	PL^2 \ge 2(a_0-1)P_0L_0F + \sum_{i=1}^{N} a_i (P_0L_{i-1}+P_0L_i)F + \sum_{i=1}^{N} P_0(L'_{i-1}+L'_i)Z_i.
	$$
\end{prop}

\begin{proof}
	The proof is similar to the one before, so we sketch it here. For any $0 \le i \le N-1$, we have
	\begin{eqnarray*}
		P_0L'^2_i - P_0L'^2_{i+1} & = & P_0(L'_i - L'_{i+1})(L'_i + L'_{i+1}) \\
		& = & P_0(a_{i+1}F + Z_{i+1})(L'_i + L'_{i+1}).
	\end{eqnarray*}
	As in the proof of Proposition \ref{intersection31}, we can check that the following three (in)equalities hold:
	\begin{eqnarray*}
		P_0(L'_i + L'_{i+1}) F & = & P_0(L_i + L_{i+1}) F, \\
		P_0L'^2_N & \ge & 0, \\
		P_0L^2_0-P_0L'^2_0 & = & 2P_0(a_0-1)L_0F.
	\end{eqnarray*}
	Hence the proposition follows immediately.
\end{proof}

\begin{prop} \label{sumai}
	Let $P$ be a nef $\QQ$-divisor on $X$ such that $P \ge L$. Then 
	$$
	P^3 \ge \left(\sum_{i=0}^{N} a_i -1 \right) P^2F.
	$$
\end{prop}

\begin{proof}
	By the decomposition in Theorem \ref{filtration}, we have
	$$
	L_0 \ge \left(\sum_{i=0}^{N} a_i -1 \right) F + L'_N.
	$$
	Let $P_0 = \sigma^* P$. It follows that
	$$
	P^3 = P_0^3 \ge P_0^2 L_0 \ge \left(\sum_{i=0}^{N} a_i -1 \right) P_0^2F = \left(\sum_{i=0}^{N} a_i -1 \right) P^2F.
	$$
	Thus the proof is completed.
\end{proof}


\section{Linear system on $2$-towers of curve fibrations}

Within this section, we assume that $X$ is a smooth $3$-fold defined over $k$ which satisfies the following commutative diagram:
$$
\xymatrix{
	X  \ar[r]^{f} \ar[rd]_{\pi} & Y \ar[d]^{h} \\
	& B
}
$$
Here $\pi: X \to B$ is a fibration from $X$ to a smooth curve $B$ with a general fiber $F$ integral, $f: X \to Y$ is a fibration from $X$ to a smooth surface $Y$ such that the general fiber $C$ is a smooth curve of genus $g \ge 2$, and $h: Y \to B$ is a surface fibration from $Y$ to $B$ with smooth general fibers. In other words, $\pi: X \to B$ is a $2$-tower of curve fibrations.

Notice that $f|_F$ induces a fibered structure on $F$ and $f(F)$ is in fact a general fiber of $h$. Therefore, we may assume that $C$ is also a general fiber of $f|_F$ on $F$ and $f(F)$ is smooth.

\subsection{Set up}
Suppose that $L \ge 0$ is a nef divisor on $X$. In the following, we construct a double filtration of $L$ with respect to the fibrations $\pi$ and $f|_F$. All notation here will be used throughout this section.

\subsubsection{The double filtration}
By Theorem \ref{filtration}, there is a birational morphism $\sigma_X: X_L \to X$ such that we have a sequence of triples
$$
\{(L_i, Z_i, a_i), \quad i=0, 1, \cdots, N\}
$$
on $X_L$ which satisfies the conditions therein. Let $f_L: X_L \stackrel{\sigma_X}{\to} X \stackrel{f}{\to} Y$ and $\pi_L: X_L \stackrel{\sigma_X}{\to} X \stackrel{\pi}{\to} B$ be the induced fibrations. By abuse of the notation, we still denote by $C$ a general fiber of $f_L$ and by $F$ a general fiber of $\pi_L$. Write 
$$
L'_i = L_i - (a_i-1)F.
$$
By (\ref{nefthresholddef}), $L'_i$ is nef. Moreover, we know that
\begin{equation}
	L'_i = a_{i+1}F + Z_{i+1} + L'_{i+1}. \label{li'}
\end{equation}

Let $\sigma_F : \tilde{F} \to F$ be a resolution of singularities on $F$ whose existence is ensured by \cite{Lipman_Resolution}. Then $\tilde{f}: \tilde{F} \stackrel{\sigma_F}{\to} F \to f(F)$ is a fibration of $\tilde{F}$. Denote by $\tilde{C}$ a general fiber of $\tilde{f}$. Then $\tilde{C}$ is also smooth and $g(\tilde{C}) = g$. 

For each $0 \le i \le N$, write $L_{i, 0} = \sigma^*_F (L_i|_F)$. By Theorem \ref{filtration2}, for any $i$, we have a sequences of triples
$$
\{(L_{i, j}, Z_{i, j}, a_{i, j}), \quad j=0, 1, \cdots, N_i\}
$$
on $\tilde{F}$ satisfying the conditions therein. Write
$$
L'_{i, j} = L_{i, j} - (a_{i, j} - 1) \tilde{C}.
$$
Similarly, $L'_{i, j}$ is nef and
\begin{equation}
	L'_{i, j} = a_{i, j+1} \tilde{C} + Z_{i, j+1} + L'_{i, j+1}. \label{lij'}
\end{equation}

\subsubsection{Horizontal base loci}
For any $1 \le i \le N$, we can decompose $Z_i$ as 
$$
Z_i = Z_i^H + Z_i^V,
$$
where $Z_i^H$ and $Z_i^V$ are respectively the horizontal and vertical part of $Z_i$ with respect to $f_L: X_L \to Y$. Let $\lambda_1 < \lambda_2 < \cdots < \lambda_l$ be all indices in $\{1, \cdots, N \}$ such that 
$$
Z_{\lambda_q}^H > 0
$$
for $1 \le q \le l$ and denote 
$$
\delta_q = \deg(Z_{\lambda_q}|_C).
$$
For our convenience, we denote $\lambda_0 = 0$ and $\lambda_{l+1} = N+1$.

\subsection{Numerical inequalities}
In this subsection, we prove some numerical results related to the horizontal base loci.
\begin{lemma} \label{horizontal1}
	Let $D \ge 0$ be a horizontal divisor on $X_L$ with respect to $f_L$ and $\delta = \deg(D|_C)$. Then for any $0 \le i \le N$, we have
	$$
	L_iDF \ge \delta (a_{i, 0} - 1). 
	$$
\end{lemma}

\begin{proof}
	Let $\tilde{D} = \sigma_F^*(D|_F)$. Then 
	$$
	\tilde{D}\tilde{C} = \deg(D|_C) = \delta.
	$$
	Therefore, it follows from (\ref{lij'}) that
	$$
	L_i D F = L_{i, 0} \tilde{D} = L'_{i, 0} \tilde{D} + (a_{i, 0} - 1) \tilde{C} \tilde{D} \ge \delta(a_{i, 0}-1).
	$$
	The proof is completed.
\end{proof}

\begin{lemma} \label{horizontal2}
	Let $D \ge 0$ be a horizontal divisor on $X_L$ with respect to $f_L$ and $\delta = \deg(D|_C)$ For any $0 \le q \le l$, if $\lambda_{q+1}-\lambda_q > 1$, then 
	$$
	L'^2_{\lambda_q} D \ge 2 \sum_{i=\lambda_q+1}^{\lambda_{q+1}-1} a_i L_i D F \ge 2 \delta \sum_{i=\lambda_q+1}^{\lambda_{q+1}-1} a_i (a_{i, 0} - 1).
	$$
\end{lemma}

\begin{proof}
	The key point here is that $Z_{i}$ is vertical with respect to $f_L$ for any $\lambda_q + 1 \le i \le \lambda_{q+1} - 1 $. It means that $Z_i$ and $D$ intersect with each other properly. In particular, for any nef divisor $M$ on $X_L$, the intersection number
	$$
	MDZ_i \ge 0.
	$$
	Combine this observation with (\ref{li'}). Then we get
	\begin{eqnarray*}
		L'^2_{i-1} D & = & L'_{i-1} (a_{i}F + Z_{i} + L'_{i}) D \\
		& \ge & L'_{i-1} (a_iF+L'_{i}) D \\
		& = & (a_{i}F + Z_{i} + L'_{i}) (a_iF+L'_{i}) D \\
		& \ge & 2 a_{i} L'_{i} D F + L'^2_{i} D.
	\end{eqnarray*}
	Summing over all above $i$, we obtain
	$$
	L'^2_{\lambda_q} D \ge 2 \sum_{i=\lambda_q+1}^{\lambda_{q+1}-1} a_i L_i D F + L'^2_{\lambda_{q+1}-1}D \ge 2 \sum_{i=\lambda_q+1}^{\lambda_{q+1}-1} a_i L_i D F.
	$$
	Hence the first inequality is proved. The second one is straightforward by applying Lemma \ref{horizontal1} to each $L_iDF$.
\end{proof}

\begin{lemma} \label{lambdaq}
	For any $1 \le q \le l$, we have
	$$
	(L^2_{\lambda_q-1}+L_{\lambda_q-1}L_{\lambda_q}+L_{\lambda_q}^2)F \ge 3L^2_{\lambda_q}F + 2 \delta_q(a_{\lambda_q, 0} - 1).
	$$
\end{lemma}

\begin{proof}
	By the decomposition in Theorem \ref{filtration}, we know that
	$$
	L_{\lambda_q-1}L_{\lambda_q}F =  \left(L_{\lambda_q} + Z_{\lambda_q}+a_{\lambda_q}F \right)L_{\lambda_q}F = L^2_{\lambda_q}F + L_{\lambda_q}Z_{\lambda_q}F.
	$$
	Together with Proposition \ref{intersectionnumber}, the above equality also gives 
	$$
	L^2_{\lambda_q-1}F \ge L_{\lambda_q-1}L_{\lambda_q}F = L^2_{\lambda_q}F + L_{\lambda_q}Z_{\lambda_q}F.
	$$
	On the other hand, by Lemma \ref{horizontal1},
	$$
	L_{\lambda_q} Z_{\lambda_q} F \ge L_{\lambda_q} Z^H_{\lambda_q} F \ge \delta_q(a_{\lambda_q, 0} - 1).
	$$
	The proof is completed by combining the three (in)equalities together.
\end{proof}

\begin{prop} \label{lambda}
	For any $1 \le q \le l$, we have 
	$$
	\sum_{i=\lambda_q}^{\lambda_{q+1}-1} a_i(L^2_{i-1} + L_{i-1}L_{i} + L^2_i)F + L'^2_{\lambda_q} Z_{\lambda_q}
	\ge \sum_{i=\lambda_q}^{\lambda_{q+1}-1} a_i \left( 3L^2_iF + 2\delta_q (a_{i, 0} - 1)\right).
	$$
\end{prop}

\begin{proof}
	If $\lambda_{q+1} - 1 = \lambda_q$, this result is implied by Lemma \ref{lambdaq} as $L'^2_{\lambda_q} Z_{\lambda_q} \ge 0$. Thus in the following, we assume that $\lambda_{q+1}-1 > \lambda_q$. By Proposition \ref{intersectionnumber}, we have 
	$$
	\sum_{i=\lambda_q+1}^{\lambda_{q+1}-1} a_i(L^2_{i-1} + L_{i-1}L_{i} + L^2_i)F \ge 3 \sum_{i=\lambda_q+1}^{\lambda_{q+1}-1} a_i L^2_iF.
	$$
	Also, Lemma \ref{horizontal2} implies that
	$$
	L'^2_{\lambda_q} Z_{\lambda_q} \ge L'^2_{\lambda_q} Z^H_{\lambda_q} \ge 2 \delta_q \sum_{i=\lambda_q+1}^{\lambda_{q+1}-1} a_i (a_{i, 0}-1).
	$$
	Combine the above two inequalities with Lemma \ref{lambdaq}, and the proof is completed.
\end{proof}

\subsection{Generalization of Proposition \ref{pl}} In this subsection, we generalize Proposition \ref{pl} to $2$-towers of curve fibrations.

\begin{prop} \label{pl2}
	Suppose that $P$ is a nef $\QQ$-divisor on $X$ such that $P \ge L$ and $P|_C \ge K_C$. Then
	$$
	PL^2 \ge (4g-4) \sum_{i=0}^{N} a_i \left(\sum_{j=0}^{N_i} a_{i, j} - 1 \right) - 2 PLF.
	$$
	Moreover, if $L|_C = K_C$, then
	$$
	PL^2 \ge L^3 \ge 3 \sum_{i=0}^{\lambda_1-1} a_i L^2_{i, 0} + (4g-4) \sum_{i=\lambda_1}^{N} a_i \left(\sum_{j=0}^{N_i} a_{i, j} - 1 \right) - 3PLF.
	$$
\end{prop}

\begin{proof}
	Write $P_0 = \sigma^*_X P$, and $\tilde{P} = \sigma^*_F (P_0|_F)$. By Proposition \ref{intersection32} and Proposition \ref{intersectionnumber}, we have
	\begin{eqnarray*}
		PL^2 & \ge & 2(a_0-1)P_0L_0F + 2 \sum_{i=1}^N a_iP_0(L_{i-1}+L_i)F \\
		& \ge & 2 \sum_{i=0}^N a_i \tilde{P} L_{i, 0} - 2 PLF.
	\end{eqnarray*}
	For any $0 \le i \le N$, by the decomposition in Theorem \ref{filtration2}, we have
	$$
	L_{i, 0} \ge \left(\sum_{j=0}^{N_i} a_{i, j} - 1\right) \tilde{C} + L'_{i, N_i}
	$$
	on $\tilde{F}$, from which we obtain
	$$
	\tilde{P} L_{i, 0} \ge \left(\sum_{j=0}^{N_i} a_{i, j} - 1\right) \tilde{P} \tilde{C} \ge (2g-2) \left(\sum_{j=0}^{N_i} a_{i, j} - 1\right).
	$$
	Combine the above results together, and the proof of the first inequality is completed.
	
	Now let us prove the second inequality. Suppose that $L|_C = K_C$. Since $Z_i$ is vertical with respect to $f_L$ for any $0 \le i \le \lambda_1-1$, we know that
	$$
	L'_{\lambda_1 - 1}|_C = K_C.
	$$
	Write $\tilde{L}_{\lambda_1-1} = \sigma^*_F (L'_{\lambda_1 - 1}|_F)$. Notice that $\nti_{f_L} (L'_{\lambda_1 - 1}) = 1$. Apply the proof of the first inequality verbatim, and it follows that
	\begin{eqnarray*}
		L'^3_{\lambda_1-1} & \ge & 2 (\nti_{f_L}(L'_{\lambda_1 - 1}) - 1) L'^2_{\lambda_1-1} F + 2 \sum_{i=\lambda_1}^{N} a_i L'_{\lambda_1-1}(L_{i-1}+L_i) F \\
		& \ge & 2 \sum_{i=\lambda_1}^{N} a_i \left(\sum_{j=0}^{N_i} a_{i, j} - 1 \right) \tilde{L}_{\lambda_1-1} \tilde{C} \\
		& \ge & (4g-4)\sum_{i=\lambda_1}^{N} a_i \left(\sum_{j=0}^{N_i} a_{i, j} - 1 \right).
	\end{eqnarray*}
	On the other hand, by Proposition \ref{intersection31}, we have
	$$
	L^3 - L'^3_{\lambda_1 - 1} \ge 3 \sum_{i=0}^{\lambda_1-1} a_i L^2_{i, 0} - 3L^2F.
	$$
	Combine the above two inequalities together, and the result follows as $L^2F \le PLF$ and $L^3 \le PL^2$.
\end{proof}

\subsection{Main result}
Our main result in this section is the following.

\begin{theorem} \label{3fold}
	Suppose that $|L|$ is base point free and $0 < L|_C \le K_C$. Let $P$ be a nef $\QQ$-divisor on $X$ such that $P \ge L$, $P|_C \ge K_C$ and $P^2F > 0$. Then
	$$
	h^0(L) \le  \left( \frac{1}{12} + \frac{1}{6g-6} \right) PL^2 + \left(\frac{P^3}{P^2F} + 1 \right) g + P^2F.
	$$
\end{theorem}

\begin{proof}
	By Proposition \ref{h0}, we know that
	\begin{equation}
	h^0(L) \le \sum_{i=0}^{N} a_i h^0(L_i|_F) \le \sum_{i=0}^{N} a_i h^0(L_{i, 0}). \label{3foldh0}
	\end{equation}
	In the following, we divide our proof into two cases.
	
	\textbf{Case 1: $L|_C < K_C$}. Apply Proposition \ref{intersection31}, and we obtain
	\begin{eqnarray}
		L^3 & \ge & 3 \sum_{i=0}^{\lambda_1-1} a_i L^2_{i, 0} - 3L^2F + \label{3foldl3} \\
		& & \sum_{q=1}^{l+1} \left(\sum_{i=\lambda_q}^{\lambda_{q+1}-1} a_i(L^2_{i-1} + L_{i-1}L_{i} + L^2_i)F + L'^2_{\lambda_q} Z_{\lambda_q}\right). \nonumber 
	\end{eqnarray}
	By Proposition \ref{lambda}, for any $1 \le q \le l$, we have 
	\begin{eqnarray}
	& & \sum_{i=\lambda_q}^{\lambda_{q+1}-1} a_i(L^2_{i-1} + L_{i-1}L_{i} + L^2_i)F + L'^2_{\lambda_q} Z_{\lambda_q} \label{3foldlambdaq} \\ 
	& \ge & \sum_{i=\lambda_q}^{\lambda_{q+1}-1} 3 a_i \left(L^2_{i, 0} + \frac{2}{3} \delta_q (a_{i, 0} - 1) \right) \nonumber
	\end{eqnarray}
	 In the following, we will focus on the difference between $PL^2$ and $L^3$.
	
	Since $P|_C \ge K_C > L|_C$, we know that $(P-L)|_C > 0$. In particular, the horizontal part of $P-L$ with respect to $f$ is non-empty, which we denote by $D$. Write
	$$
	\delta_0 = \deg(D|_C).
	$$ 
	Let $D_0$ be the proper transform of $D$ by $\sigma_X$. Then $D_0$ is still horizontal with respect to $f_L$ and $\deg(D_0|_C) = \delta_0$. By Lemma \ref{horizontal2}, if $\lambda_1 > 1$, then 
	\begin{equation}
	L'^2_0 D_0 \ge 2 \delta_0\sum_{i=1}^{\lambda_1-1} a_i(a_{i, 0} -1). \label{l'2d}
	\end{equation}
	Furthermore, by Lemma \ref{horizontal1}, we know that
	\begin{eqnarray}
	(L^2_0 - L'^2_0) D_0 & = & (a_0 - 1) (L_0 + L'_0)D_0F \label{l2d-l'2d} \\
	& = & 2a_0L_0D_0F - 2L_0D_0F \nonumber\\
	& \ge & 2 \delta_0a_0(a_{0, 0}-1) - 2(P-L)LF. \nonumber
	\end{eqnarray}
	Combine (\ref{l'2d}) and (\ref{l2d-l'2d}), and we obtain
	\begin{equation}
		L^2 D = L^2_0 D_0 \ge 2 \delta_0 \sum_{i=0}^{\lambda_1-1} a_i(a_{i, 0} -1) - 2PLF + 2L^2F \label{3foldl2d}
	\end{equation}
	when $\lambda_1 > 1$. On the other hand, if $\lambda_1 = 1$, it can be easily checked that (\ref{3foldl2d}) follows from (\ref{l2d-l'2d}) directly as $L'^2_0 D_0 \ge 0$.
	
	Now by (\ref{3foldl3}), (\ref{3foldlambdaq}) and (\ref{3foldl2d}), it follows that
	\begin{eqnarray}
		PL^2 & \ge & L^3 + L^2 D \label{3foldpl2} \\
		& \ge & \sum_{q=0}^{l} \sum_{i=\lambda_q}^{\lambda_{q+1} - 1} 3a_i \left(L^2_{i, 0} + \frac{2}{3} \delta_q (a_{i, 0} - 1) \right) - 3PLF. \nonumber
	\end{eqnarray}
	Combine (\ref{3foldh0}) and (\ref{3foldpl2}), and apply  Proposition \ref{sumai} and Proposition \ref{comparisonl<2g-2} (which we will prove later) together with the first inequality in Proposition \ref{pl2}. It follows that
	\begin{eqnarray*}
		h^0(L) - \frac{1}{12} PL^2 & \le & \sum_{q=0}^{l} \sum_{i= \lambda_q}^{\lambda_{q+1} - 1} a_i \left(h^0(L_{i, 0}) - \frac{1}{4}L^2_{i, 0} - \frac{1}{6} \delta_q(a_{i, 0}-1)\right) + \frac{1}{4} PLF \\
		& \le & \frac{2}{3} \sum_{i=0}^{N} a_i \left(\sum_{j=0}^{N} a_{i, j} - 1 \right) + g \sum_{i=0}^{N} a_i +  \frac{1}{4} PLF \\
		& \le & \frac{1}{6g-6} (PL^2 + 2PLF) +  g \sum_{i=0}^{N} a_i + \frac{1}{4} PLF \\
		& \le & \frac{1}{6g-6} PL^2 + \left(\frac{P^3}{P^2F} + 1 \right) g + PLF.
	\end{eqnarray*}
	Since $PLF \le P^2F$, we finish the proof in this case.
	
	\textbf{Case 2: $L|_C = K_C$}. It is easy to check that (\ref{3foldl3}) and (\ref{3foldlambdaq}) still hold in this case. Together with the fact that $PL^2 \ge L^3$ and $PLF \ge L^2F$, we obtain
	\begin{equation}
		PL^2 \ge 3 \sum_{i=0}^{\lambda_1 - 1} a_i L^2_{i, 0} + \sum_{q=1}^{l} \sum_{i=\lambda_q}^{\lambda_{q+1} - 1} 3a_i \left(L^2_{i, 0} + \frac{2}{3} \delta_q (a_{i, 0} - 1) \right) - 3PLF. \label{case2l3}
	\end{equation}
	By Proposition \ref{comparisonl<2g-2}, for any $1 \le q \le l$ and $\lambda_q \le i \le \lambda_{q+1}-1$, 
	\begin{equation}
		h^0(L_{i, 0}) \le \frac{1}{4} L^2_{i, 0} + \frac{1}{6} \delta_{q}(a_{i, 0}-1) + \frac{2}{3} \left(\sum_{j=0}^{N_i} a_{i, j} - 1 \right) + g. \label{case2<2g-2}
	\end{equation}
	Furthermore, notice that $L_i = K|_C$ for $0 \le i \le \lambda_1-1$. For any such $i$, by \cite[Theorem 1.2]{YuanZhang_RelNoether}, we have
	\begin{equation}
		h^0(L_{i, 0}) \le \frac{1}{4}L^2_{i, 0} + \frac{1}{4g-4} L^2_{i, 0} + g. \label{case2=2g-2}
	\end{equation}
	Combine (\ref{case2l3})--(\ref{case2=2g-2}) and the second inequality in Proposition \ref{pl2}. Similar to Case 1, it follows that
	\begin{eqnarray*}
		h^0(L) - \frac{1}{12} PL^2 & \le &  \sum_{i=0}^{\lambda_1-1} a_i \left(h^0(L_{i, 0}) - \frac{1}{4} L^2_{i, 0}\right) + \frac{1}{4} PLF + \\
		& & \sum_{q=1}^{l} \sum_{i= \lambda_q}^{\lambda_{q+1} - 1} a_i \left(h^0(L_{i, 0}) - \frac{1}{4}L^2_{i, 0} - \frac{1}{6} \delta_q(a_{i, 0}-1)\right) \\ 
		& \le &  \sum_{i=0}^{\lambda_1-1} \frac{1}{4g-4} a_i L^2_{i, 0} + \frac{2}{3} \sum_{i=\lambda_1}^{N} a_i \left(\sum_{j=0}^{N} a_{i, j} - 1 \right) + g \sum_{i=0}^{N} a_i + \frac{1}{4} PLF \\
		& \le & \frac{1}{6g-6} (PL^2+3PLF) + g \sum_{i=0}^{N} a_i + \frac{1}{4} PLF \\
		& \le & \frac{1}{6g-6} PL^2 + \left(\frac{P^3}{P^2F} + 1 \right) g + PLF. 
	\end{eqnarray*}
	Again, we have $PLF \le P^2F$. Thus the whole proof is completed once we verify the following proposition.
\end{proof}

\begin{prop} \label{comparisonl<2g-2}
	For any $1 \le q \le l$ and any $\lambda_q \le i \le \lambda_{q+1} - 1$, we have
	$$
	h^0(L_{i, 0}) \le \frac{1}{4} L^2_{i, 0} + \frac{1}{6} \delta_{q}(a_{i, 0}-1) + \frac{2}{3} \left(\sum_{j=0}^{N_i} a_{i, j} - 1 \right) + g.
	$$
	Moreover, if we are in Case 1 of the proof of Theorem \ref{3fold}, then the above inequality still holds for any $0 \le i \le \lambda_1 - 1$ if replacing $\delta_q$ by $\delta_0$ therein.
\end{prop}

\begin{proof}
	Let us first assume that $C$ is hyperelliptic. 
	By Theorem \ref{surface} (1) and (3), we know that for any $0 \le i \le N$,
	\begin{eqnarray*}
		h^0(L_{i, 0})  & \le & \frac{1}{4} L^2_{i, 0} + a_{i, 0} + \frac{1}{2} \sum_{j=1}^{N_i} a_{i, j} + (g-1) \\
		& \le & \frac{1}{4} L^2_{i, 0} + \frac{1}{3} (a_{i, 0} - 1) + \frac{2}{3} \left( \sum_{j=0}^{N_i} a_{i, j} - 1\right) + g.
	\end{eqnarray*}
	On the other hand, since $|L_i|_C|$ is base point free,  $\deg(L_i|_C) < 2g-2$ must be even for any $i \ge \lambda_1$. In particular, it implies that
	$$
	\delta_q \ge 2
	$$
	for any $q \ge 1$. Moreover, if we are in Case 1 of the proof of Theorem \ref{3fold}, then $\deg(L|_C) < 2g-2$ is also even, and we still have
	$$
	\delta_0 = \deg(P|_C) - \deg(L|_C) \ge (2g-2) - (2g-4) \ge 2.
	$$
	Thus the proof is completed in the hyperelliptic case.
	
	Now assume that $C$ is non-hyperelliptic. Recall that $L_i|_C < K_C$ for any $i \ge \lambda_1$. If $\deg(L_i|_C) > 0$, then by Theorem \ref{surface} (4),
	$$
	h^0(L_{i, 0}) \le  \frac{1}{4} L^2_{i, 0} + \left(\frac{1}{2} a_{i, 0} + \frac{1}{4} a_{i, N_i} - 1 \right) + g,
	$$
	which is stronger than the desired inequality. Moreover, if we are in Case 1 of the proof of Theorem \ref{3fold}, then for any $0 \le i \le \lambda_1-1$, we have $0 < L_i|_C < K_C$. Thus the above inequality also holds for $0 \le i \le \lambda_1 - 1$, so does the proposition.
	
	If $\deg(L_i|_C) = 0$, then we deduce that
	$$
	2g-2 \ge \delta_q \ge 3.
	$$
	Otherwise $|L_{\lambda_q-1}|_C|$ is base point free and $\deg(L_{\lambda_q-1}|_C) \le 2$, which would imply that $C$ is hyperelliptic. Therefore, by Theorem \ref{surface} (1), we obtain
	$$
	h^0(L_{i, 0}) \le a_{i, 0} \le \frac{1}{6} \delta_q(a_{i, 0} - 1) + \frac{1}{2} (a_{i, 0} - 1) + g.
	$$
	Hence the whole proof is completed.
\end{proof}

\begin{remark}
	From the above proof, it is easy to see that the inequality in Proposition \ref{comparisonl<2g-2} can be improved if $C$ is non-hyperelliptic. This improvement will result in a better inequality than that in Theorem \ref{3fold}, which we will treat in a subsequential paper.
\end{remark}

\section{Linear system on fibered $3$-folds over surfaces}

Let $f: X \to Y$ be a fibration from a normal $3$-fold $X$ to a smooth surface $Y$ defined over $k$ with general fiber $C$ a smooth curve of genus $g \ge 2$. Fix a smooth and very ample divisor $H \ge 0$ on $Y$ with genus $g(H) \ge 2$ ($H$ is indeed a smooth curve on $Y$). Write $G = f^*H$. 

The main theorem in this section is the following:

\begin{theorem} \label{relnoether32}
	Suppose that $L$ is a nef $\QQ$-divisor on $X$ such that 
	$L|_C \ge K_C \ge \rounddown{L|_C}$.
	\begin{itemize}
		\item[(1)] If $L^2G > 0$, then 
		$$
		h^0(\CO_X(\rounddown{L})) \le \left(\frac{1}{12} + \frac{1}{6g-6} \right) L^3 + g \left(\frac{L^3}{L^2G} + 1\right) + L^2G.
		$$
		\item[(2)] If $L^2G = 0$, then
		$$
		h^0(\CO_X(\rounddown{L})) \le g.
		$$
	\end{itemize}
\end{theorem}

\begin{proof}
	Without loss of the generality, we may assume that $L \ge 0$.
	Replacing $X$ by an appropriate resolution, we may further assume that there is a decomposition
	$$
	L = M + Z
	$$
	on $X$ such that
	\begin{itemize}
		\item $|M|$ is base point free;
		\item $Z \ge 0$ is $\QQ$-divisor;
		\item $h^0(M) = h^0(\CO_X(\rounddown{L}))$.
	\end{itemize}
	In fact, we may even assume that $X$ is smooth as the resolution of singularities in dimension three is also known in positive characteristics \cite{Cossart_Piltant_Resolution1,Cossart_Piltant_Resolution2,Cutkosky_Resolution}. In particular, we have
	$$
	M|_C \le \rounddown{L}|_C \le \rounddown{L|_C} \le K_C.
	$$
	
	In order to prove (1), we first assume that $\deg(M|_C) > 0$. Now choose a general pencil in $|H|$ and denote by $\sigma_Y: Y' \to Y$ the blowing up of the indeterminacies of this pencil. Then $Y'$ is a fibered surface over $\mathbb{P}^1$. Let $X'= X \times_Y Y'$. Then we have the following diagram:
	$$
	\xymatrix{
		X' \ar[r]^{\sigma_X} \ar[d]_{f'} \ar@/_2pc/[dd]_{\pi} & X \ar[d]^{f} \\
		Y' \ar[r]^{\sigma_Y} \ar[d]_{h} & Y \\
		\mathbb{P}^1
	}
	$$
	In particular, $\pi: X' \to \mathbb{P}^1$ is a $2$-tower of curve fibrations. Let $F$ be a general fiber of $\pi$. Then $F=\sigma^*_X G-E_X$ where $E_X$ is the exceptional divisor of $\sigma_X$. For simplicity, we still denote by $C$ a general fiber of $f'$. 
	
	Write $L' = \sigma^*_X L$ and $M'=\sigma^*_X M$. By our assumption, $|M'|$ is base point free and $\deg(M'|_C) > 0$. More importantly, $L'^2F = L^2G > 0$. Therefore, by Theorem \ref{3fold}, we have
	$$
	h^0(M') \le \left(\frac{1}{12} + \frac{1}{6g-6} \right) L'M'^2 + g \left(\frac{L'^3}{L'^2F} + 1 \right) + L'^2F,
	$$
	which implies that
	$$
	h^0(\CO_X(\rounddown{L})) \le \left(\frac{1}{12} + \frac{1}{6g-6} \right) L^3 + g \left(\frac{L^3}{L^2G} + 1 \right) + L^2G.
	$$
	
	Second, we assume that $\deg(M|_C) = 0$. Since $|M|$ is base point free, there is a divisor $D \ge 0$ on $Y$ such that $|D|$ is base point free and $M = f^*D$. In particular, as $f$ has connected fibers, we deduce that
	$$
	h^0(M) = h^0(D).
	$$
	If $D=0$, then
	$$
	h^0(M) = h^0(D) = 1,
	$$
	and the inequality in (1) still holds. If $D \ne 0$, we have $DH > 0$ as $H$ is very ample. In this case, we use the fibration $h$ to estimate $h^0(D)$. Denote by $H'$ a general fiber of $h$. Then $H' = \sigma^*_Y H - E_Y$ where $E_Y$ is the exceptional divisor of $\sigma_Y$. Denote $D'= \sigma^*_Y D$. Then $D'H' = DH > 0$. Since $g(H')=g(H) \ge 2$, by Theorem \ref{surface} (2), 
	$$
	h^0(D') \le \frac{1}{2} D'^2 + D'H',
	$$
	which implies that
	$$
	h^0(M) \le \frac{1}{2} D^2 + DH \le \frac{1}{4g-4} \left(LM^2 + 2 LMG\right) \le \frac{1}{4g-4} \left(L^3 + 2 L^2 G \right).
	$$
	Since we always have
	$$
	\frac{1}{4g-4} \le \frac{1}{12} + \frac{1}{6g-6} \le \frac{1}{4}
	$$
	when $g \ge 2$, the inequality in (1) still holds in this case.

	Now let us prove (2). In this case, we deduce that $M-G$ can not be effective. Otherwise, by Proposition \ref{intersectionnumber}, we would have
	$$
	L^2G \ge LMG \ge LG^2 = \deg(L|_C) H^2 \ge (2g-2) H^2 > 0.
	$$
	This is a contradiction. Therefore, from the long exact sequence
	$$
	0 \to H^0(M-G) \to H^0(M) \to H^0(M|_G),
	$$
	we know that
	$$
	h^0(M) \le h^0(M|_G).
	$$
	Notice that $G$ is fibered over $H$ with a general fiber $C$. We further deduce that $M|_G-C$ can not be effective. Otherwise, we would have
	$$
	L^2G \ge LMG = (L|_G) (M|_G) \ge (L|_G)C \ge 2g-2>0,
	$$
	which is again a contradiction. In a similar way, we can get
	$$
	h^0(M|_G) \le h^0(M|_C) \le g.
	$$
	Hence the whole proof is completed.
\end{proof}


\section{$F$-stable dimension and the Chern character of vector bundles}
Our main goal in this section is to set up the relation between $h^0_F(\CE)$ and $\ch_2(\CE)$ for certain vector bundle $\CE$ over surfaces in positive characteristics. Throughout this section, we assume that $\charr k = p > 0$.

\subsection{$F$-stable dimension of global sections}
Let $V$ be a variety defined over $k$. Let $F^e$ be the $e^{\rm th}$ absolute Frobenius morphism of $V$. In Definition \ref{h0f}, for any coherent sheaf $\mathcal E$ on $V$, the \emph{$F$-stable dimension of global sections} $h^0_F(\mathcal{E})$ of $\mathcal E$ is defined as
$$
h^0_F(\mathcal E) := \liminf_{e \to \infty} \frac{h^0(F^{e*} \mathcal E)}{p^{e \dim V}}.
$$

In fact, under certain circumstances, the invariant $h^0_F$ is closely related to some well-understood invariants.

\begin{prop} \label{h0fl}
	Let $V$ be a variety of dimension $n$ over $k$. Let $\CL$ be an ample line bundle on $V$. Then 
	$$
	h^0_F(\mathcal{L}) = \lim\limits_{e \to \infty} \frac{h^0(F^{e*}\mathcal{L})}{p^{ne}} = \frac{1}{n!} c^n_1(\CL) = \ch_n(\CL).
	$$
\end{prop}

\begin{proof}
	The result is essentially from the Riemann-Roch theorem and the Serre vanishing theorem. Notice that $F^{e*}\CL = \CL^{\otimes p^e}$. When $e$ is sufficiently large, the Serre vanishing theorem asserts that
	$$
	h^i(F^{e*}\mathcal{L}) = 0, \quad \forall i>0.
	$$
	Therefore, by the Riemann-Roch theorem,
	$$
	h^0(F^{e*}\mathcal{L}) = \frac{p^{ne}}{n!} c^n_1(\CL) + o(p^{ne}).
	$$
	The result follows by letting $e \to \infty$.
\end{proof}

In certain sense, Proposition \ref{h0fl} suggests that there may be certain relation between $h^0_F(\CE)$ and $\ch_n(\CE)$ when $\CE$ is ``positive". In the following, we mainly focus on this relation when $\dim V = 2$.

\subsection{Semi-positive vector bundles over curves}
Let $C$ is a smooth curve of genus $g$ over $k$. We have the following lemma which is definitely known to the experts.

\begin{lemma} \label{vbcurve}
	For any semi-positive vector bundle $\CE$ of rank $r$ over $C$,
	$$
	h^1(\CE) \le 2rg.
	$$
\end{lemma}

\begin{proof}
	By Serre duality, it is equivalent to show that
	$$
	h^0(\CO_C(K_C) \otimes \CE^\vee) \le 2rg.
	$$
	Notice that if $C$ is isomorphic to $\mathbb{P}^1$, then $\CO_C(-K_C) \otimes \CE$ is positive. Hence 
	$$
	h^0(\CO_C(K_C) \otimes \CE^\vee) = 0,
	$$
	and the result holds in this case.
	
	Now suppose that $g > 0$. Take a divisor $\Sigma = q_0 + q_1 + \ldots + q_{2g-2}$ on $C$ consisting of $2g-1$ (distinct) points. Consider the following exact sequence
	$$
	0 \to H^0(\CO_C(K_C-\Sigma) \otimes \CE^\vee) \to H^0(\CO_C(K_C) \otimes \CE^\vee) 
	\to H^0(\CO_\Sigma(K_C) \otimes \CE^\vee|_\Sigma).
	$$
	Since $\CO_C(\Sigma-K_C) \otimes \CE$ is positive, we know that
	$$
	h^0(\CO_C(K_C-\Sigma) \otimes \CE^\vee) = 0.
	$$
	Therefore, we deduce that 
	$$
	h^0(\CO_C(K_C) \otimes \CE^\vee) \le h^0(\CO_\Sigma(K_C) \otimes \CE^\vee|_\Sigma) \le r (2g-1).
	$$
	The proof is completed.
\end{proof}

\subsection{Semi-positive vector bundles over surfaces}
Let $Y$ be a smooth surface over $k$. In the following, we generalize Lemma \ref{vbcurve} to semi-positive vector bundles over $Y$.

\begin{lemma} \label{vbsurface}
	There exists a smooth curve $C$ on $Y$ such that  
	$$
	h^2(\mathcal{E}) \le 2rg(C)
	$$ 
	for any semi-positive vector bundle $\CE$ of rank $r$ over $Y$.
\end{lemma}

\begin{proof}
	Let $\CE$ be any semi-positive vector bundle $\CE$ of rank $r$ over $Y$. By Serre duality again, it suffices to show that 
	$$
	h^0(\CO_Y(K_Y) \otimes \mathcal{E}^\vee) \le 2rg(C)
	$$
	for certain $C$ on $Y$ which is independent of $\CE$.
	
	Fix an ample divisor $A$ on $Y$ such that $A+K_Y \ge 0$ is ample. Replacing $A$ by its tensor power if necessary, we can find a smooth curve $C \in |A+K_Y|$ on $Y$ by Bertini's theorem. Now we have the exact sequence
	$$
	0 \to \CO_Y(-A) \otimes \mathcal{E}^\vee \to \CO_Y(K_Y) \otimes \mathcal{E}^\vee \to \CO_C(K_Y) \otimes (\mathcal{E}|_C)^\vee \to 0.
	$$
	Notice that the bundle $\CO_Y(A) \otimes \mathcal{E}$ is positive. In particular, 
	$$
	h^0(\CO_Y(-A) \otimes \mathcal{E}^\vee) = 0.
	$$
	Therefore, we deduce that 
	$$
	h^0(\CO_Y(K_Y) \otimes \mathcal{E}^\vee) \le h^0(\CO_C(K_Y) \otimes (\mathcal{E}|_C)^\vee) \le h^0(\CO_C(K_Y+C) \otimes (\mathcal{E}|_C)^\vee).
	$$
	Moreover, by the adjunction and Serre duality, we have 
	$$
	h^0(\CO_C(K_Y+C) \otimes (\mathcal{E}|_C)^\vee)  = h^0(\CO_C(K_C) \otimes (\mathcal{E}|_C)^\vee) = h^1(\CE|_C).
	$$
	Since $\CE|_C$ is semi-positive of rank $r$ over $C$, by Lemma \ref{vbcurve},
	$$
	h^1(\CE|_C) \le 2rg(C).
	$$
	Hence the proof is completed.
\end{proof}

The main result here is the following proposition.
\begin{prop} \label{h0ch2}
	Let $\CE$ be a semi-positive vector bundle over $Y$. Then
	$$
	h^0_F(\mathcal E) \ge \ch_2(\mathcal E).
	$$
\end{prop}

\begin{proof}
	By the Riemann-Roch theorem for $\CE$, when $e$ is sufficiently large, we have
	$$
	h^0(F^{e*} \mathcal E) + h^2(F^{e*} \mathcal E) \ge \chi(F^{e*} \mathcal E) =  p^{2e} \ch_2(\CE) + o(p^{2e}).
	$$
	Notice that $F^{e*} \mathcal E$ is also semi-positive over $Y$. By Lemma \ref{vbsurface}, we know that $h^2(F^{e*} \mathcal E)$ is uniformly bounded above by a constant depending only on $Y$. Hence the result follows by letting $e \to \infty$.
\end{proof}

\begin{remark} \label{h0ch1}
	Let $\mathcal{E}$ be any vector bundle of rank $r$, not necessarily semi-positive, over a smooth curve $C$. Observe that
	$$
	h^0(F^{e*} \mathcal{E}) \ge \deg(F^{e*} \mathcal{E}) + r\chi(\CO_C) = p^e \deg(\mathcal{E}) + r\chi(\CO_C).
	$$
	We immediately obtain that
	$$
	h^0_F(\mathcal{E}) \ge \deg(\mathcal{E}) = \ch_1(\mathcal{E}).
	$$
\end{remark}

\subsection{Almost semi-positive vector bundles over surfaces}
In this subsection, we generalize Proposition \ref{h0ch2} to vector bundles over $Y$ which are very close to be semi-positive. Here $Y$ still denotes a smooth surface over $k$.

For the convenience of the reader, we recall some notation that have been introduced in Section $2$. For a vector bundle $\CE$ over $Y$,  there is the projection map $\pi_{\CE}: \mathbb{P}(\CE) \to Y$. Let $H_{\CE}$ be a divisor associated to $\CO_{\mathbb{P}(\CE)} (1)$. In the following, $A \ge 0$ denotes a very ample divisor on $Y$ and $m$ is a positive integer. 
\begin{lemma} \label{h2bound}
	Let $\mathcal{E}$ be a vector bundle of rank $r$ over $Y$ such that $H_{\CE} + \frac{1}{m} \pi^*_{\CE} A$ is nef. Then for any $e > 0$, we have
	$$
	h^2(F^{e*} \CE) \le \left(\frac{p^e}{m} + 1\right)^2 r A^2 + \left(\frac{p^e}{m} + 1\right) r K_Y A + c,
	$$
	where $c$ is a constant only depending on $r$ and $Y$.
\end{lemma}

\begin{proof}
	For any $e > 0$, we denote $F^{e*} \CE$ by $\CE^{(e)}$ for simplicity. Then we have the following commutative diagram:
	$$
	\xymatrix{
		\mathbb{P}(\CE^{(e)})  \ar@{->}[r]^{F^e_{\mathbb{P}}} \ar@{->}[d]_{\pi_{\CE^{(e)}}} & \mathbb{P}(\CE)  \ar@{->}[d]^{\pi_{\CE}} \\
		Y \ar@{->}[r]^{F^e} & Y
	}
	$$
	In particular, we have $F^{e*}_\PP(\pi^*_\CE A) = p^e \pi^*_{\CE^{(e)}} A$ as $F^{e*} A = p^e A$.
	
	By our assumption, $H_{\CE} + \frac{1}{n} \pi^*_{\CE} A$ is nef. By pulling back via $F^e_{\PP}$, we know that $H_{\CE^{(e)}} + \frac{p^e}{m} \pi^*_{\CE^{(e)}} A$ is nef. Let $\alpha_e = \rounddown{\frac{p^e}{m}} + 1$. Then $H_{\CE^{(e)}} + \alpha_e \pi^*_{\CE^{(e)}} A$ is ample, which implies that $\CE^{(e)} \otimes \CO_Y(\alpha_e A)$ is positive.
	
	Choose a smooth curve $C \in |\alpha_e A|$. Then we have the exact sequence
	$$
	0 \to \CE^{(e)} \to \CE^{(e)}  \otimes \CO_Y(\alpha_e A) \to \CE^{(e)}|_C \otimes \CO_C(\alpha_e A)  \to 0.
	$$
	Taking the cohomology, we obtain
	$$
	H^1(\CE^{(e)}|_C \otimes \CO_C(\alpha_e A)) \to H^2 (\CE^{(e)}) \to H^2 (\CE^{(e)} \otimes \CO_Y(\alpha_e A)) \to 0.
	$$
	By Lemma \ref{vbsurface}, we can find a number $c'>0$ independent on $\CE$ and $e$ such that 
	$$
	h^2 (\CE^{(e)} \otimes \CO_Y(\alpha_e A)) \le c'.
	$$
	Furthermore, by Lemma \ref{vbcurve} and the adjunction formula, we know that 
	$$
	h^1(\CE^{(e)}|_C \otimes \CO_C(\alpha_e A)) \le 2r g(C) = r(\alpha_e^2 A^2 + \alpha_e K_YA + 2).
	$$
	Combine the above two results and set $c=c'+2r$. Then the proof is completed.
\end{proof}

\begin{remark}
	The author was informed by Professor Adrian Langer that in \cite[Theorem 4.1]{Langer_Moduli}, a result similar to Lemma \ref{h2bound} was proved.
\end{remark}

Finally, we get the following proposition.

\begin{prop} \label{h0ch2weak}
	Let $\mathcal{E}$ be a vector bundle of rank $r$ on $Y$ such that $H_{\CE} + \frac{1}{m} \pi^*_{\CE} A$ is nef. Then
	$$
	h^0_F(\CE) \ge \ch_2(\CE) - \frac{r}{m^2}A^2.
	$$
\end{prop}

\begin{proof}
	The proof is very similar to that of Proposition \ref{h0ch2}. In fact, for $e$ sufficiently large, the Riemann-Roch theorem for $F^{e*}\CE$ tells us that
	$$
	h^0(F^{e*}\CE) + h^2(F^{e*}\CE) \ge p^{2e} \ch_2(\CE) + o(p^{2e}).
	$$
	By Lemma \ref{h2bound}, we know that
	$$
	h^2(F^{e*}\CE) \le \frac{r}{m^2} A^2 p^{2e} + o(p^{2e}).
	$$
	Take the limit as $e \to \infty$, and the proof is completed.
\end{proof}


\section{Proof of Theorem \ref{main}}

Now we are ready to prove Theorem \ref{main}. We will prove the characteristic $p$ part first and then prove Theorem \ref{main} in characteristic zero by the mod $p$ reduction. 

In this section, $f: X \to Y$ is a fibration from a normal $3$-fold $X$ to a smooth surface $Y$ defined over $k$ with general fiber $C$ a smooth curve of genus $g \ge 2$.

\subsection{Theorem \ref{main} in positive characteristics}
Within this subsection, we assume that $\charr k = p > 0$. We first state one more general result.

\begin{theorem} \label{maingeneral}
	Let $L$ be a nef $\QQ$-divisor on $X$ such that $L|_C \ge K_C \ge \rounddown{L|_C}$. Then for any $\QQ$-Cartier Weil divisor $M \le \rounddown{L}$, we have
	$$
	\left(\frac{1}{12} + \frac{1}{6g-6} \right) L^3 \ge h^0_F(f_* \CO_X(M)).
	$$
	Moreover, if $f_* \CO_X(M)$ is locally free and semi-positive, then
	$$
	\left(\frac{1}{12} + \frac{1}{6g-6} \right) L^3 \ge \ch_2(f_* \CO_X(M)).
	$$
\end{theorem}

\begin{proof}
	Notice that by Proposition \ref{h0ch2}, the second result is a direct consequence of the first one. Therefore, we only need to prove the first inequality.
	
	Let $F^e: Y \to Y$ be the $e^{\rm th}$ absolute Frobenius morphism of $Y$. Then we have the following diagram:
	$$
	\xymatrix{
		X'_e \ar@{->}[r]^\sigma \ar@{->}[dr]_{f'_e}& X_e \ar@{->}[r]^{F^e_{X/Y}} \ar@{->}[d]^{f_e} & X \ar@{->}[d]^f \\
		& Y \ar@{->}[r]^{F^e} & Y
	}
	$$
	Here $X_e = X \times_{F^e} Y$, $F^e_{X/Y}$ is the $e^{\mathrm{th}}$ relative Frobenius morphism of $X$ over $Y$ and $\sigma: X'_e \to X_e$ denotes the normalization of $X_e$. For simplicity, we still denote by $C$ a general fiber of $f'_e$.
	
	In the following, we denote
	$$
	L_e = F^{e*}_{X/Y} L, \quad L'_e = \sigma^*L_e 
	\quad \mbox{and} \quad M_e = F^{e*}_{X/Y} M.
	$$ 
	Fix a smooth and very ample divisor $H \ge 0$ on $Y$ with $g(H) \ge 2$. Write 
	$$
	G = f^*H, \quad G_e=f^*_eH \quad \mbox{and} \quad G'_e =f'^*_e H.
	$$
	Notice that we always have
	\begin{equation}
	{F^e}^* H = p^eH. \label{frob}
	\end{equation}
	All the above yield 
	\begin{equation}
	L'^3_e = L^3_e = p^{2e} L^3 \quad \mbox{and} 
	\quad L'^2_e G'_e = L^2_eG_e = p^eL^2G. \label{increasing}
	\end{equation}
	On the other hand, by our assumption, $L'_e|_C = L|_C$ for any $e>0$. It implies that $L'_e|_C \ge K_C \ge  \rounddown{L'_e|_C}$. Another important fact is that since $\CO_{X_e}$ is a subsheaf of $\sigma_* \CO_{X'_e}$, by the projection formula, we deduce that
	\begin{equation}
		h^0(\CO_{X_e}(M_e)) \le h^0(\CO_{X_e}(\rounddown{L_e})) \le h^0(\CO_{X'_e}(\rounddown{L'_e})). \label{h0comparison}
	\end{equation}
	
	Now we divide the proof into two cases. First, assume that $L^2G > 0$. By (\ref{increasing}), this implies that $L'^2_e G'_e > 0$ for any $e > 0$. Therefore, by Theorem \ref{relnoether32} (1) for $L'_e$, it follows that
	$$
	h^0(\CO_{X'_e}(\rounddown{L'_e})) \le \left(\frac{1}{12} + \frac{1}{6g-6}\right) L'^3_e + g \left(\frac{L'^3_e}{L'^2_eG'_e} + 1 \right) + L'^2_e G'_e,
	$$
	which implies that
	$$
	\frac{h^0({f_e}_*\CO_{X_e}(M_e))}{p^{2e}} = \frac{h^0(\CO_{X_e}(M_e))}{p^{2e}} \le \frac{h^0(\CO_{X_e}(\rounddown{L_e}))}{p^{2e}} \le \left(\frac{1}{12} + \frac{1}{6g-6}\right) L^3 + o(1)
	$$
	due to (\ref{increasing}) and (\ref{h0comparison}). Second, if $L^2G = 0$, similarly, we know that $L'^2_eG'_e = 0$ for any $e > 0$. By Theorem \ref{relnoether32} (2), 
	$$
	h^0(\CO_{X'_e}(\rounddown{L'_e})) \le g,
	$$
	which implies that
	$$
	h^0({f_e}_*\CO_{X_e}(M_e)) = h^0(\CO_{X_e}(M_e)) \le g
	$$
	again by (\ref{h0comparison}).
	
	Finally, letting $e \to \infty$, we obtain that
	$$
	\left(\frac{1}{12} + \frac{1}{6g-6}\right) L^3 \ge  \liminf_{e \to \infty} \frac{h^0({f_e}_*\CO_{X_e}(M_e))}{p^{2e}} =\liminf_{e \to \infty} \frac{h^0(F^{e*}f_* \CO_X(M))}{p^{2e}}.
	$$
	Thus the proof is completed.
\end{proof}	
	
Now, we can prove Theorem \ref{main} in positive characteristics.

\begin{theorem} [Theorem \ref{main} in positive characteristics] \label{mainp}
	Suppose that $f$ is a relatively minimal fibration of curves of genus $g \ge 2$ over $k$ as in Definition \ref{nfoldfibration}. Then
	$$
	\left(\frac{1}{12} + \frac{1}{6g-6} \right) K^3_{X/Y} \ge h^0_F(f_* \omega_{X/Y}).
	$$
	Moreover, if $f_* \omega_{X/Y}$ is semi-positive, then 
	$$
	\left(\frac{1}{12} + \frac{1}{6g-6} \right) K^3_{X/Y} \ge \ch_2(f_* \omega_{X/Y}).
	$$
\end{theorem}

\begin{proof}
	We only need to replace both $L$ and $M$ in Theorem \ref{maingeneral} by $K_{X/Y}$ to get these two inequalities here.
\end{proof}

\subsection{Theorem \ref{main} in characteristic zero}
In this subsection, let $f$ be a relatively minimal fibration of curves of genus $g \ge 2$ defined over $k$ of characteristic zero as in Definition \ref{nfoldfibration}.

We first prove a weaker theorem.

\begin{theorem} \label{weakmain}
	Let $A \ge 0$ be an ample $\QQ$-divisor on $X$ such that 
	$\deg(A|_C) < 1$. Then
	$$
	\left(\frac{1}{12} + \frac{1}{6g-6}\right)(K_{X/Y} + A)^3 \ge \ch_2(f_* \omega_{X/Y}).
	$$
\end{theorem}

\begin{proof}
	By Lefschetz principle, we can assume that $k$ is a finite extension over $\QQ$. Let $\CZ$ be a scheme of finite type over $\ZZ$ with the function field $k$. Let $\CX \to \CY \to \CZ$ be a projective and flat morphism extending $X \to Y \to \spec k$. Replacing $\mathcal{Z}$ by its Zariski open set and taking a finite and \'etale cover, we may assume that the generic fiber $X \to Y \to \spec k$ has an integral model $\mathcal{X} \to \mathcal{Y} \to \mathcal{Z}$ such that 
	\begin{itemize}
		\item [(i)] $\CY \to \CZ$ is a proper and smooth morphism;
		\item [(ii)] $f: \CX \to \CY$ is a proper, flat and Cohen-Macaulay morphism of pure relative dimension one;
		\item [(iii)] The divisor $A$ extends to a $\QQ$-Cartier divisor $\CA$ on $\CX$.
	\end{itemize}
	By \cite[Theorem 3.5.1, 3.6.1]{Conrad_Duality} again, the relative dualizing sheaf $\omega_{\CX/\CY}$ is well-defined under this setting. It is $\CY$-flat and compatible with base changes. Therefore, we may assume that $f_*\omega_{\CX/\CY}$ is also locally free. In particular, for any closed point $z \in \CZ$, we have
	$$
	\ch_2({f_z}_* \omega_{\CX_z / \CY_z}) = \ch_2(f_* \omega_{X/Y}).
	$$
	On the other hand, since both $A$ and $K_{X/Y}+A$ are ample, according to \cite[Chap. III, Theorem 4.7.1]{Grothendieck_EGA3}, by further shrinking $\CZ$ if necessary, we may assume that 
	\begin{itemize}
		\item [(iv)] both $\CA_z$ and $K_{\CX_z/\CY_z}+\CA_z$ are ample $\QQ$-divisors on $\CX_z$ for any closed point $z \in \CZ$.
	\end{itemize}
	Moreover, if $C_z$ denotes a general fiber of $f_z: \CX_z \to \CY_z$, then $\deg(\CA_z|_{C_z}) < 1$. In particular, $\rounddown{\CA_z|_{C_z}} = 0$. Hence 
	$$
	K_{\CX_z/\CY_z}|_{C_z}+\CA_z|_{C_z} \ge K_{C_z} = \rounddown{K_{\CX_z/\CY_z}|_{C_z}+\CA_z|_{C_z}}.
	$$
	
	Recall that there is the projection map $\pi: = \pi_{f_* \omega_{X/Y}}: \mathbb{P}(f_*{\omega_{X/Y}}) \to Y$ and a Cartier divisor $H:=H_{f_* \omega_{X/Y}}$ associated to $\CO_{f_* \omega_{X/Y}}(1)$. Then both $\pi$ and $H$ extend to a universal projection $\pi: \PP(f_* \omega_{\CX/\CY}) \to \CY$ over $\CY$ and a Cartier divisor $\CH$ associated to $\CO_{\PP(f_* \omega_{\CX/\CY})}(1)$. 
	
	Pick a very ample divisor $A' \ge 0$ on $Y$. Since $f_* \omega_{X/Y}$ is semi-positive \cite{Viehweg_Weak_Positivity}, we know that $H$ is nef. In particular, for any $m > 0$, the $\QQ$-divisor $H+\frac{1}{m}\pi^*A'$ on $\PP(f_* \omega_{X/Y})$ is ample. Fix an $m > 0$, by \cite[Chap. III, Theorem 4.7.1]{Grothendieck_EGA3} again, we may shrink $\CZ$ once more such that 
	\begin{itemize}
		\item [(v)] $A'$ extends to a Cartier divisor $\CA'$ on $\CY$;
		\item [(vi)] for any closed point $z \in \CZ$, $\CA'_z$ is very ample on $\CY_z$ and $\CH_z + \frac{1}{m} \pi_z^* \CA'_z$ is an ample $\QQ$-divisor on $\PP({f_z}_* \omega_{\CX_z/\CY_z})$.
	\end{itemize}
	
	From now on, we assume that (i)--(vi) hold. Take any closed point $z \in \CZ$. Then we get a fibration $f_z: \CX_z \to \CY_z$ of curves of genus $g \ge 2$ over the residue field of $z$ of positive characteristic as in Definition \ref{nfoldfibration}, not necessarily relatively minimal. Let $F: \CY_z \to \CY_z$ be the absolute Frobenius morphism of $\CY_z$. By Theorem \ref{maingeneral}, we know that 
	$$
	\left(\frac{1}{12} + \frac{1}{6g-6} \right)(K_{\CX_z / \CY_z}+\CA_z)^3 \ge h^0_F({f_z}_* \omega_{\CX_z/\CY_z}).
	$$
	On the other hand, as the rank of ${f_z}_*\omega_{\CX_z/\CY_z}$ is $g$, by Proposition \ref{h0ch2weak}, 
	$$
	h^0_F({f_z}_* \omega_{\CX_z/\CY_z}) \ge \ch_2 ({f_z}_*\omega_{\CX_z/\CY_z}) - \frac{g}{m^2} \CA'^2_z.
	$$
	It follows that
	$$
	\left(\frac{1}{12} + \frac{1}{6g-6} \right) (K_{\CX_z / \CY_z}+\CA_z)^3 \ge \ch_2 ({f_z}_* \omega_{\CX_z/\CY_z}) - \frac{g}{m^2} \CA'^2_z,
	$$
	which implies that for $f: X \to Y$, we have
	$$
	\left(\frac{1}{12} + \frac{1}{6g-6} \right) (K_{X/Y} + A)^3 \ge \ch_2(f_* \omega_{X/Y}) - \frac{g}{m^2} A'^2.
	$$
	Notice that $m$ is arbitrary. Thus the proof is completed.
\end{proof}

Now we can finish the whole proof of Theorem \ref{main}.
\begin{proof} [Proof of Theorem \ref{main} in characteristic zero]
	Choose an ample divisor $A \ge 0$ on $X$. Then for any integer $n > \deg(A|_C)$, we have $\deg \left( \left. \left( \frac{1}{n} A \right) \right|_C \right) < 1$. Therefore, for any such $n$, by Theorem \ref{weakmain}, we have
	$$
	\left(\frac{1}{12} + \frac{1}{6g-6} \right) \left(K_{X/Y} + \frac{1}{n} A\right)^3 \ge \ch_2(f_*\omega_{X/Y}).
	$$
	Thus the proof is completed by letting $n \to \infty$.
\end{proof}

\subsection{An example}
In this subsection, we construct a $3$-fold fibration of curves of genus $g \ge 2$ which indicates that the inequality in Theorem \ref{main} is very close to be optimal.

Let $\pi: X \to \PP^1 \times \PP^2$ be a double cover defined over $k$ branched along a smooth and even divisor $B \equiv 2L$, where $\CO_X(L) = p^*_1 \CO_{\PP^1}(g+1) \otimes p^*_2 \CO_{\PP^2}(1)$. Here $p_1$ and $p_2$ are the canonical projections. That is, we have the following diagram:
$$
\xymatrix{
	X \ar[rr]^\pi \ar[rrd]_f & & \PP^1 \times \PP^2  \ar[d]^{p_2} \ar[rr] ^{p_1} & & \PP^1 \\
	& & \PP^2 
	}
$$
It is easy to see that $X$ is smooth and $f: X \to \PP^2$ is a fibration of curves of genus $g$ as the general fiber of $f$ is a double cover over $\PP^1$ ramified at $2g+2$ distinct points.

By the Riemann-Hurwitz formula,
$$
\omega_X = \pi^* \left(\omega_{\PP^1 \times \PP^2} \otimes \CO_{\PP^1 \times \PP^2}(L)\right) = \pi^* \left( p_1^* \CO_{\PP^1}(g-1) \otimes p_2^* \CO_{\PP^2} (-2)\right).
$$
Therefore, we deduce that
$$
\omega_{X/{\PP^2}} = \omega_X \otimes f^* \CO_{\PP^2}(3) = \pi^*\left( p^*_1 \CO_{\PP^1}(g-1)  \otimes p_2^* \CO_{\PP^2}(1) \right).
$$
In particular, $K_{X/{\PP^2}}$ is ample, which implies that $f$ is relatively minimal.

\begin{prop} \label{main=}
	Let $f: X \to \PP^2$ be the above fibration. Then
	$$
	K^3_{X/{\PP^2}} = \frac{12g-12}{g} \ch_2(f_*\omega_{X/{\PP^2}}) = 6g-6.
	$$
\end{prop}

\begin{proof}
	This is a direct computation. It is easy to see that
	$$
	c_1^3\left(p^*_1 \CO_{\PP^1}(g-1)  \otimes p_2^* \CO_{\PP^2}(1)\right) = 3g-3.
	$$
	By the above formula of $\omega_{X/{\PP^2}}$, we obtain that
	$$
	K^3_{X/{\PP^2}} = (\deg \pi) c_1^3\left(p^*_1 \CO_{\PP^1}(g-1) \right) = 6g-6.
	$$
	
	On the other hand, by the above commutative diagram and the projection formula, we have
	\begin{eqnarray*}
		f_* \omega_{X/{\PP^2}} & = & f_* \left( \pi^*\left( p^*_1 \CO_{\PP^1}(g-1)  \otimes p_2^* \CO_{\PP^2}(1) \right) \right) \\ 
		& = & f_* \left(\pi^*p_1^* \CO_{\PP_1}(g-1) \otimes f^* \CO_{\PP^2}(1)\right) \\
		& = & p_{2*} p_1^* \CO_{\PP^1}(g-1) \otimes \CO_{\PP^2}(1) \\
		& = & \left(\bigoplus_{i=1}^g \CO_{\PP^2} \right) \otimes \CO_{\PP^2}(1) \\
		& = & \bigoplus_{i=1}^g \CO_{\PP^2}(1).
	\end{eqnarray*}
	From this, we deduce that
	$$
	c_1^2 (f_* \omega_{X/{\PP^2}}) = g^2 \quad \mbox{and} \quad c_2(f_* \omega_{X/{\PP^2}}) = \frac{g(g-1)}{2}.
	$$
	Therefore, we have
	$$
	\ch_2(f_* \omega_{X/{\PP^2}}) = \frac{1}{2}c_1^2 (f_* \omega_{X/{\PP^2}}) - c_2(f_* \omega_{X/{\PP^2}}) = \frac{g}{2}.
	$$
	Hence the proof is completed.
\end{proof}

Furthermore, if $\charr k = p>0$, we can compute the invariant $h^0_F(f_* \omega_{X/{\PP^2}})$ in this case.

\begin{prop} \label{h0=ch2}
	Let $f: X \to \PP^2$ be the above fibration defined over $k$ of characteristic $p > 0$. Then
	$$
	h^0_F(f_* \omega_{X/{\PP^2}}) = \ch_2(f_* \omega_{X/{\PP^2}}) = \frac{g}{2}.
	$$
\end{prop}

\begin{proof}
	From the proof of Proposition \ref{main=}, we know that
	$$
	f_* \omega_{X/{\PP^2}} = \bigoplus_{i=1}^g \CO_{\PP^2}(1).
	$$
	Let $F^e$ be the $e^{\mathrm{th}}$ absolute Frobenius morphism of $\PP^2$ over $k$. Then
	$$
	h^0(F^{e*}f_* \omega_{X/{\PP^2}}) = h^0 \left(\bigoplus_{i=1}^g \CO_{\PP^2}(p^e)\right) = \frac{(p^e+2)(p^e+1) g}{2}.
	$$
	Thus by Definition \ref{h0f}, we obtain
	$$
	h^0_F(F^{e*}f_* \omega_{X/{\PP^2}}) = \liminf_{e \to \infty} \frac{h^0_F(F^{e*}f_* \omega_{X/{\PP^2}})}{p^{2e}} = \frac{g}{2}.
	$$
	Therefore, the proof is completed by the result in Proposition \ref{main=} about $\ch_2(f_* \omega_{X/{\PP^2}})$.
\end{proof}

In fact, this proposition also indicates that the inequality in Proposition \ref{h0ch2} is sharp.

\section{Slope inequality (\ref{surfaceslope}) revisited}

In the last section, we illustrate how the method (in fact, a simpler version) introduced in this paper gives a new proof of the slope inequality (\ref{surfaceslope}) of Cornalba-Harris and Xiao in arbitrary characteristic.\footnote{The author was asked by Professor Xiaotao Sun about a proof of  (\ref{surfaceslope}) using characteristic $p > 0$ method in 2013.}

Let $f: X \to Y$ be a fibration from a smooth surface $X$ to a smooth curve $Y$ defined over $k$ with general fiber $C$ a smooth curve of genus $g \ge 2$. We observe that it suffices to prove the following result which is a surface version of Theorem \ref{maingeneral}:

\begin{theorem} \label{mainsurface}
	Let $L$ be a nef $\QQ$-divisor on $X$ defined over $k$ of characteristic $p>0$ such that $L|_C \ge K_C \ge  \rounddown{L|_C}$. Then for any divisor $M \le \rounddown{L}$, we have
	$$
	\left(\frac{1}{4} + \frac{1}{4g-4}\right) L^2 \ge h^0_F(f_* \CO_X(M)).
	$$
\end{theorem}

With this result, to prove (\ref{surfaceslope}) in characteristic $p > 0$, we only need to replace both $L$ and $M$ by $K_{X/Y}$ and apply Remark \ref{h0ch1}. When $\charr k = 0$, using the same argument as in the proof of Theorem \ref{weakmain}, we can prove that 
$$
\left(\frac{1}{4} + \frac{1}{4g-4}\right) (K_{X/Y} + A)^2 \ge \ch_1(f_* \omega_{X/Y}) = \deg f_* \omega_{X/Y}
$$
for any ample $\QQ$-divisor $A$ on $X$ with $\deg(A|_C) < 1$, provided that $f$ is relatively minimal. Notice that the corresponding argument is much simpler in the current case. In particular, we do not need to perturb $f_*\omega_{X/Y}$ by ample $\QQ$-divisors, because Remark \ref{h0ch1} is a positivity-free result. Then (\ref{surfaceslope}) in this case follows from a limiting argument on $A$ similar to the proof of Theorem \ref{main}. We leave the proof to the interested reader.

In the following, we prove Theorem \ref{mainsurface}.

\begin{proof}[Proof of Theorem \ref{mainsurface}]
	In fact, it is enough to prove that
	\begin{equation} \label{weakmainsurface}
		h^0(M) \le \left(\frac{1}{4} + \frac{1}{4g-4}\right) L^2 + \frac{LC+2}{2}.
	\end{equation}
	Once this is verified, we will obtain Theorem \ref{mainsurface} using the Frobenius limiting trick presented in the proof of Theorem \ref{maingeneral}. Again, the argument in the surface case is dramatically simpler as the error terms remain the same when applying the Frobenius limiting trick and we do not need to choose the very ample divisor $H$ on $Y$.
	
	Since $LC \ge 2g-2$, we may assume that $h^0(M) \ge 2$ in (\ref{weakmainsurface}). Replacing $X$ by an appropriate blowing up, we may further assume that $|M|$ is base point free and $M|_C \le K_C$. If $M|_C = 0$, then $M \sim_{\mathrm{num}} rC$ for some $r \ge h^0(M) - 1$. Thus
	$$
	h^0(M) \le r + 1 \le \frac{LM}{LC} + 1 \le \frac{1}{2g-2}L^2 + 1.
	$$
	Therefore, (\ref{weakmainsurface}) is verified in this case as
	$$
	\frac{1}{2g-2} \le \frac{1}{4} + \frac{1}{4g-4}
	$$
	when $g \ge 2$. On the other hand, if $M|_C = K_C$, by \cite[Theorem 1.2]{YuanZhang_RelNoether}, we have
	$$
	h^0(M) \le \left(\frac{1}{4} + \frac{1}{4g-4}\right) M^2 + g \le \left(\frac{1}{4} + \frac{1}{4g-4}\right) L^2 + \frac{LC + 2}{2}.
	$$
	
	In the following, we assume that $0 \ne M|_C < K_C$. By Theorem \ref{filtration2}, we obtain triples
	$$
	\{(M_i, Z_i, a_i) | i = 0, \ldots, N \}
	$$
	on $X$ which satisfy the conditions therein. If $C$ is hyperelliptic, by Theorem \ref{surface} (3), 
	$$
	h^0(M) - \frac{1}{4} M^2 \le a_0 + \frac{1}{2} \sum_{i=1}^{N} a_i + \frac{MC}{2}.
	$$
	On the other hand, notice that $\deg(M|_C)$ is even, which gives $(L-M)C \ge 2$. Since $M':=M-(a_0-1)C$ is also nef, we obtain that
	$$
	(L-M)M \ge (a_0-1)(L-M)C + (L-M)M' \ge 2a_0 + (L-M)C.
	$$
	Combine the above two inequalities and apply Proposition \ref{pl}. It follows that
	\begin{eqnarray*}
		h^0(M) - \frac{1}{4} LM & \le & h^0(M) - \frac{1}{4} M^2 - \frac{1}{4}(L-M)M \\
		& \le & \frac{1}{2}  \sum_{i=0}^{N} a_i + \frac{MC}{2} + \frac{(L-M)C}{4} \\
		& \le & \frac{1}{4g-4} LM + \frac{LC+2}{2}.
	\end{eqnarray*}
	If $C$ is non-hyperelliptic, just by Theorem \ref{surface} (4) and Proposition \ref{pl}, we deduce that
	$$
	h^0(M) - \frac{1}{4} M^2\le \frac{1}{2} \sum_{i=0}^{N} a_i + \frac{MC}{2} \le \frac{1}{4g-4} LM + \frac{LC + 2}{2}.
	$$
	As a result, (\ref{weakmainsurface}) is proved in this case as $M^2 \le LM \le L^2$.
\end{proof}


\bibliography{Ref_Zhang}
\bibliographystyle{amsalpha}

\end{document}